\documentclass[11pt,reqno]{amsart}

\usepackage{marginnote,amssymb,mathrsfs,latexsym,verbatim,pdfsync,color,mathabx}

\usepackage [utf8]{inputenc}
\usepackage[colorlinks=true, linkcolor=blue]{hyperref}

\newcommand{\bbC}{{\mathbb C}}

\newcommand{\bbH}{{\mathbb H}}  
\newcommand{\bbN}{{\mathbb N}}
\newcommand{\bbR}{{\mathbb R}}

\newcommand{\bbZ}{{\mathbb Z}}

\def\cE{{\mathcal E}}
\def\cF{{\mathcal F}}

\def\cH{{\mathcal H}}
\def\cI{{\mathcal I}}

\def\cK{{\mathcal K}}
\def\cL{{\mathcal L}}

\def\cP{{\mathcal P}}

\def\cS{{\mathcal S}}

\def\cU{{\mathcal U}}
\def\cV{{\mathcal V}}
\def\cW{{\mathcal W}}

\def\cZ{{\mathcal Z}}

\def\sF{{\mathscr F}}

\def\sL{{\mathscr L}}

\def\sU{{\mathscr U}}

\def\Re{\operatorname{Re}}%{\, \text{\rm Re\,}}
\def\Im{\operatorname{Im}}%{\, \text{\rm Im\,}}

\def\la{\langle}
\def\ra{\rangle}

\def\eps{\varepsilon}
\def\z{\zeta} %%%%%%%%%%%%%%%%%%%%%%% NOTICE!!!! {\zeta}
\def\vp{\varphi}
\def\ov{\overline}
\def\p{\partial}

\def\ms{\medskip}

\def\endpf{\medskip\hfill $\Box$

\smallskip %\ms

}  
\def\epf{\endpf}

\def\Hol{\operatorname{Hol}}
\def\tr{\operatorname{tr}}

\def\HS{\operatorname{HS}}
\def\ran{\operatorname{ran}}
\def\supp{\operatorname{supp}}
\def\sinc{\operatorname{sinc}}

\def\PW{\cP\cW}

\newtheorem{thm}{Theorem}[section]
\newtheorem{prop}[thm]{Proposition}
\newtheorem{cor}[thm]{Corollary}
\newtheorem{lem}[thm]{Lemma}
\newtheorem{defn}[thm]{Definition}
\newtheorem{remark}[thm]{Remark}

\begin{document}

\title[Sampling of entire functions of exponential type in
  $\bbC^{n+1}$]{Sampling in spaces of entire functions \\ of exponential type in
  $\bbC^{n+1}$} 
\author[A. Monguzzi, M. M. Peloso, M. Salvatori]{Alessandro Monguzzi*, 
Marco M. Peloso**, Maura Salvatori**}

 \address{*Dipartimento di Matematica e Applicazioni, Universit\`a
   degli Studi di Milano--Bicocca, Via R. Cozzi 55, 20126 Milano,
   Italy}   
\email{{\tt alessandro.monguzzi@unimib.it, alessandro.monguzzi@unibg.it}}
 \address{* Dipartimento di Ingegneria Gestionale, dell'Informazione e della Produzione, Universit\`a
   degli Studi di Bergamo, Viale G. Marconi 5, 24044 Dalmine,
   Italy}   
\address{**Dipartimento di Matematica, Universit\`a degli Studi di
  Milano, Via C. Saldini 50, 20133 Milano, Italy}
\email{{\tt marco.peloso@unimi.it}}
\email{{\tt maura.salvatori@unimi.it}} 
\thanks{{\em Math Subject Classification 2020} 32A15, 32A37, 32A50, 46E22.}
\thanks{The authors are members of the
  Gruppo Nazionale per l'Analisi Matematica, la Probabilit\`a e le
  loro 
Applicazioni (GNAMPA) of the Istituto Nazionale di Alta Matematica (INdAM) }
\thanks{The first author is partially supported by the 2020
  INdAM--GNAMPA grant  \emph{Alla frontiera tra l'analisi complessa in
    pi\`u variabili e l'analisi armonica}. The second and
    third authors are  partially supported by the 2020
   INdAM--GNAMPA grant {\em Fractional Laplacians and subLaplacians on Lie groups and trees}.}
 \begin{abstract}
In this paper we consider the question of sampling for spaces of
entire functions of exponential type in several variables.  The
novelty resides in the growth
condition we impose on the entire functions, that is, that their restriction to a hypersurface  is square integrable with
respect to a natural measure.  The hypersurface we consider is the
boundary $b\cU$ of the Siegel upper half-space $\cU$ and it is
fundamental that $b\cU$ can be identified with the Heisenberg group
$\bbH_n$. We consider entire functions in $\bbC^{n+1}$ of exponential 
type with respect to the hypersurface $b\cU$
whose restriction to $b\cU$ are square integrable with respect to the
Haar measure on $\bbH_n$. For these functions we prove  a version
of the Whittaker--Kotelnikov--Shannon Theorem.
Instrumental in our work are spaces of 
entire functions in $\bbC^{n+1}$ of exponential 
type with respect to the hypersurface $b\cU$ whose
restrictions to $b\cU$
belong to some
homogeneous Sobolev space on $\bbH_n$.  For these spaces, using
the group Fourier transform on $\bbH_n$, we prove a Paley--Wiener type
theorem  and a Plancherel--P\'olya type inequality. 
\end{abstract}
\maketitle

\section{Introduction and statement of the main results}

The  classical Paley--Wiener theorem characterizes the entire functions of
exponential type $a$ in the complex plane whose restriction to the
real line is square integrable as the space of $L^2$-functions whose
Fourier transform is supported in the interval $[-a,a]$.  For such
functions perhaps the most far reaching result is the
Whittaker--Kotelnikov--Shannon sampling theorem.
These results have been extended to several
variables for functions in $\bbC^n$ whose restrictions to the surface
$\{\Im z=0\}=\bbR^n$ have Fourier transform supported in a compact set
$\Omega$, see e.g. \cite{Stein-Weiss} and \cite{Landau}.

In this paper we take a different approach. We consider a hypersurface $M$
that separates the whole space $\bbC^{n+1}$ into two unbounded
connected components and  entire 
functions in $\bbC^{n+1}$ that satisfy some exponential growth
condition adapted to $M$ -- namely, of quadratic order in the complex
tangential directions to $M$ and of linear growth in the transversal direction.
We also require that these functions have restriction to $M$ which is square integrable, or more
generally is in some homogeneous Sobolev  space. For such
functions we prove 
 a Paley--Wiener type theorem and a Whittaker--Kotelnikov--Shannon
 sampling theorem. 
\ms

Consider the complex spaces
$\bbC^{n+1}$, with $n\ge1$, and the strongly pseudoconvex hypersurface 
$$
b\cU=\bigl\{ \z=(\z',\z_{n+1})\in\bbC^n\times\bbC:\, \Im
\z_{n+1} =\textstyle{\frac14} |\z'|^2 \bigr\} \,,
$$
which is the topological boundary of the
{\em  Siegel upper-half space} 
\begin{equation}\label{U-def}
\cU =\Big\{ \z=(\z',\z_{n+1})\in\bbC^n\times\bbC:\, 
\varrho(\z):= \Im\z_{n+1} -\textstyle{\frac14} |\z'|^2 >0 \Big\} \,.
\end{equation}
 It is well known that
$\cU$ is biholomorphic to
the unit ball $B$ in $\bbC^{n+1}$.
The $(1,1)$-form $\theta=\frac i2 (\ov\p-\p)\rho$ is a
pseudo-hermitian structure on $M=b\cU$, and it is non-degenerate.
Then,
$\theta\wedge(d\theta)^n$
is a volume form on $b\cU$, i.e. $\theta$ is a contact form.  Then, there
exists a natural Riemannian metric $g_\theta$ on $b\cU$.
The boundary $b\cU$ of $\cU$ can be endowed with the structure of a
nilpotent Lie group, namely the Heisenberg group $\bbH_n$.   It turns
out that the volume form $\theta\wedge(d\theta)^n$ coincides with the
Haar measure on $\bbH_n$. The Haar measure on $\bbH_n$ coincides
with two canonical measures defined 
on the strongly pseudoconvex manifold $M$,
namely the Webster metric \cite{Dragomir} and  the
Fefferman metric \cite[p. 259]{Fefferman}, resp.
The volume forms constructed starting from such metrics coincide with
the Haar measure on $\bbH_n$. 
We also remark that the induced metric from
$\bbC^{n+1}$ differs from the metric $g_{\lambda\theta}$, for every smooth
$\lambda:b\cU\to (0,+\infty)$.

In $\bbC^{n+1}$ we
introduce 
coordinates by means of a foliation of copies
of $b\cU$.    
Given $\z=(\z',\z_{n+1})\in \bbC^{n+1}$, we define $\Psi: \bbC^n\times \bbC
\to \bbC^n\times \bbR\times\bbR$ by
\begin{equation*}
\Psi(\z',\z_{n+1})
= \big( \z', \Re \z_{n+1}, \Im \z_{n+1} -\textstyle{\frac14}|\z'|^2 \big)
=: (z,t,h) 
\,.
\end{equation*}
Then,  
$\Psi$ is a $C^\infty$-diffeomorphism,  
and
\begin{equation*}
\Psi^{-1} (z,t,h) = \bigl(z,t+i\textstyle{\frac14}|z|^2+ih\bigr) =:  (\z',\z_{n+1})
\,. 
\end{equation*}

Notice that
$h=\varrho(\z',\z_{n+1})$, where $\varrho$ is as in
\eqref{U-def}. 
Clearly, the boundary $b\cU$ is characterized by the points of
$\bbC^{n+1}$ such that $\Psi(\z)=(z,t,0)$, that is, 
$h=\varrho(\z)=0$, and 
$$
\cU = 
\big\{ \z\in\bbC^{n+1}:\, \Psi(\z)=(z,t,h)\,  \text{ is such that }\
  h>0\big\}. 
$$
When $h=0$,  we write $[z,t]$ in
place of $(z,t,0)$.  Then, the boundary $b\cU$ can be identified with
the Heisenberg group $\bbH_n$, which  is the set $\bbC^n\times\bbR$ endowed
with product 
$$
[w,s][z,t]  =  \big[ w+z, s+t -\textstyle{\frac12} \Im (w\cdot\bar z)\big] \,.
$$
Recall that  $\bbH_n$ is a nilpotent Lie group  and 
the Lebesgue measure on $\bbC^n\times\bbR$ coincides with both the
right and left Haar
measure on $\bbH_n$. 

On $\bbH_n$ we consider the standard (positive) sub-Laplacian $\Delta$ and its
fractional powers $\Delta^{s/2}$ (see Section \ref{Basic-facts-sec}
for details).
Let $\cS=\cS(\bbH_n)$ denote
the space of Schwartz functions on $\bbH_n$.  
Then, we have the following  definition. 
\begin{defn}\label{dotWps-def}{\rm
For $1<p<\infty$ and $s>0$ we define  the {\em homogeneous Sobolev space} ${\dot{W}}^{s,p}={\dot{W}}^{s,p}(\bbH_n)$ as
the completion of $\cS$ 
with respect to  the norm  $\|\Delta^{s/2} \vp\|_{L^p}$, $\vp\in\cS$.
More precisely, given the equivalence relation on the space of $L^p$-Cauchy
sequences of Schwartz functions,  
$\{ \vp_k\}\sim \{ \psi_k\} $ if  $\Delta^{s/2}(\vp_k-\psi_k)\to 0$,
as $k\to+\infty$, and denoting by
$[\{ \vp_k\}]$ the equivalence classes, then 
\begin{multline*}
{\dot{W}}^{s,p}
= \Big\{ \big[\{\vp_k\}\big]: \ \{
\vp_k\}\subseteq\cS,\, \{ \Delta^{s/2}\vp_k\} \text{\ is a Cauchy
sequence in\ } L^p, \\
\text{with } 
\big\| \big[\{\vp_k\}\big]\big\|_{{\dot{W}}^{s,p}} = \lim_{k\to+\infty}
\| \Delta^{s/2} \vp_k \|_{L^p} 
\Big\}\,.
\end{multline*}
}
\end{defn}

It is easy to characterize the homogeneous  spaces
${\dot{W}}^{s,p}(\bbH_n)$ when $1<p<\infty$ and $0<s<(2n+2)/p$, see Section \ref{Basic-facts-sec}.\footnote{We
point out that $2n+2$ is the {\em homogeneous} dimension of $\bbH_n$.}

On $\bbH_n$
we define a {\em homogeneous}
norm (with respect to the natural anisotropic dilations) by setting 
\begin{equation*}
|[z,t]| := 
\bigl(\textstyle{\frac{1}{16} } |z|^4 +t^2\bigr)^{1/4}\, .
\end{equation*}
Then, we introduce a ``$\cU$-adapted norm'' in $\bbC^{n+1}$,  
\begin{equation*} 
\| \z\|_\cU =|[z,t]|^2 +|h|\,, \qquad \text{where}\quad \Psi(\z)= (z,t,h)\,.
\end{equation*}
Notice that $\| \z\|_\cU$ grows like $|[z,t]|^2 $ in the complex
tangential directions of $\bbH_n$  and 
like $|h|$ in the transversal directions.  

We now introduce the spaces of entire functions we deal with.
For a function $F$ defined on $\bbC^{n+1}$, we set 
$\widetilde
F=F\circ\Psi^{-1}$
 and $\widetilde
F_h[z,t]:=\widetilde F(z,t,h)$.  Then, in particular, $\widetilde
F_0 = F_{|_{b\cU}}$. 

\begin{defn}{\rm
Let $a>0$ be given.  
We define the space of entire functions of exponential type $a$ \emph{with
respect to the hypersurface $b\cU$} as
$$
\cE_a = \Big\{ F\in\Hol(\bbC^{n+1}): \, \text{for every } \eps>0\
\text{there exists\ } C_\eps>0  \text{ such that \ } |F(\z)|\le C_\eps
e^{(a+\eps)\|\z\|_\cU} \Big\} \,.
$$

We define the corresponding Paley--Wiener space as
$$
\PW_a  =
\big\{ F\in\cE_a:  \,\widetilde
F_0 \in L^2(\bbH_n)\ \text{with norm\ } 
\| F\|_{\PW_a} =   \|\widetilde F_0 \|_{L^2(\bbH_n)} \big\} \,.
$$
For $0<s<n+1 
$ we  define the \emph{fractional }Paley--Wiener spaces $\PW_a^s$ as
$$
\PW_a^s  =
\big\{ F\in\cE_a:  \,
\widetilde F_0 \in \dot W^{s,2}\ \text{with norm\ } 
\| F\|_{\PW_a^s} =   \|  \widetilde  F_0 \|_{\dot W^{s,2}} \big\} \,.
$$
}
\end{defn}

We point out that the fractional Paley--Wiener spaces $\PW_a^s$ arise naturally 
 in our setting, since they constitute a tool for proving our main
 results.
 In the $1$-dimensional setting
 these spaces were introduced in \cite{MPS-frac}, see also \cite{MPS-sob}.

In order to state our main results we recall the basic facts about the Fourier transform on the Heisenberg
group (see also \cite{AMPS, Folland-phase, R}).  
For $\lambda\in\bbR^*$  the  Fock space is defined as
\begin{equation*}
\cF^\lambda 
=\bigg\{ F \in\Hol(\bbC^n):\ 
\bigg( \frac{|\lambda|}{2\pi} \bigg)^n \int_{\bbC^n} |F(z)|^2\,
e^{-\frac{|\lambda|}{2} |z|^2}dz <+\infty \bigg\}
\end{equation*}
so that $\cF^\lambda = \cF^{|\lambda|}$.  
For $\lambda\in\bbR^*$ and $[z,t]\in\bbH_n$, the
Bargmann 
representation $\beta_\lambda[z,t]$ is the operator on
$\cF^\lambda$
\begin{equation}\label{Barg-rep} 
\beta_\lambda[z,t] F (w)
= \begin{cases} 
 e^{i\lambda t-\frac\lambda2 w\cdot\ov z -\frac\lambda4 |z|^2}
F(w+z)  \quad & \text{if\ } \lambda>0 ,\smallskip\cr
e^{i\lambda t+\frac\lambda2 w\cdot z +\frac\lambda4 |z|^2}
F(w+\bar z) & \text{if\ } \lambda<0 \,.
\end{cases} 
\end{equation}
 Notice that 
$\beta_\lambda[z,t] = \beta_{-\lambda} [\ov z,-t]$, when
$\lambda<0$. 
Given $f\in L^1(\bbH_n)$ its Fourier transform consists of a family of operators $\{\beta_\lambda(f)\}_{\lambda\in\bbR^*}$ where 
$\beta_\lambda(f)$ is a Hilbert--Schmidt operator 
on $\cF^\lambda$ given by
$$
 \beta_\lambda (f) F (w) 
= \int_{\bbH_n} f[z,t] \beta_\lambda[z,t]F (w)\, dzdt\,. 
$$

If $f\in L^2(\bbH_n)$ and $\|\cdot\|_{\HS}$ denotes the
Hilbert--Schmidt norm  on $\cF^\lambda$, we have Plancherel's formula
\begin{equation}\label{Plancherel}
\| f\|_{L^2(\bbH_n)}^2 = \frac{1}{(2\pi)^{n+1} }
\int_{\bbR}  \|\beta_\lambda(f)\|_{\HS}^2 |\lambda|^n\,
d\lambda\,,
\end{equation}
and, if $f\in L^1\cap L^2(\bbH_n)$, the inversion formula
\begin{equation}\label{inve-form}
f[z,t] = \frac{1}{(2\pi)^{n+1}}  
\int_{\bbR}  \tr \big( \beta_\lambda(f)\beta_\lambda[z,t]^*\big)
|\lambda|^n\, d\lambda \,.
\end{equation}

\begin{defn}\label{F-T-A2nu}{\rm
For $s\in\bbR$ we define  the space 
$\cL^2_s $ as the space of  measurable fields  of operators 
$$\tau:
\bbR^*\to \prod_{\lambda\in\bbR^*}  \sL(\cF^\lambda)
$$
 where $\sL(\cF^\lambda)$ denotes the bounded operators on
$\cF^\lambda$,  
such that
$$
\|\tau\|_{\cL^2_s}^2: = \frac{1}{(2\pi)^{n+1}} 
\int_{\bbR^*} \|\tau(\lambda)\|_{\HS}^2 \,
|\lambda|^{n+s} d\lambda <+\infty , 
$$
where 
$\|\cdot\|_{\HS} :=\|\cdot\|_{\HS(\cF^\lambda)}$. 
We also define 
$\cH^2_s$ as the subspace of  $\tau\in \cL^2_s$ such that \smallskip 
\begin{itemize} 
\item[(i)] $\tau(\lambda)=0 $  for $\lambda>0$; \smallskip 
\item[(ii)] $\ran( \tau(\lambda)) \subseteq
   \operatorname{span}\{1\} $,  when $\lambda<0$\,.\smallskip 
 \end{itemize}
If $E\subseteq (-\infty,0)$, then $\cH^2_s(E)
= \big\{ \tau\in \cH^2_s:\, \supp \tau\subseteq E\big\}$. 
Finally, when $s=0$ we simply write $\cL^2$ and $\cH^2$
in place of  $\cL^2_0$ and $\cH^2_0$, resp.
}
\end{defn}

Our first result provide the expected characterization of the spaces $\PW_a$ and $\PW_a^s,\, 0\le s<n+1$.

\begin{thm}\label{main-4} 
Let $0\le s<n+1$. If $F\in \PW_a^s$, then 
$\beta_\lambda( \widetilde F_0) \in \cH_s^2 \big([-a,0)\big)$,
\begin{equation*}
F(\z)= \widetilde F_h[z,t] =
\frac{1}{(2\pi)^{n+1}} \int_{-a}^0
e^{\lambda  h}
\tr\big(\beta_\lambda( \widetilde F_0) \beta_\lambda[z,t]^*\big)\,
|\lambda|^n d\lambda\,, 
\end{equation*}
and 
$\| F\|_{\PW_a^s} = \| \beta_\lambda( \widetilde F_0) \|_{\cL_s^2}$.
Conversely, let $\tau\in \cH_s^2 \big([-a,0)\big)$, and define 
\begin{equation}\label{PW-D-eq3}
F(\z)= \widetilde F_h[z,t] =
\frac{1}{(2\pi)^{n+1}} \int_{-a}^0
e^{\lambda  h}
\tr\big(\tau(\lambda) \beta_\lambda[z,t]^*\big)\,
|\lambda|^n d\lambda\,.
\end{equation}
Then $F\in  \PW_a^s$, $\beta_\lambda( \widetilde F_0) =\tau(\lambda)$
and $\| F\|_{\PW_a^s} = \| \tau\|_{\cL_s^2}$.  
\end{thm}

We show that for functions in
$\PW_a$ a 
sampling result holds true,  extending the classical
Whittaker--Kotelnikov--Shannon Theorem in dimension $1$. 
We recall that given 
 a reproducing kernel Hilbert space $\cK$ of holomorphic functions on
 a domain $\Omega$, a sequence $\Gamma=\{\gamma\}\subseteq\Omega$
 is a sampling sequence for $\cK$ if there exist constants
$A,B>0$ such that for all $F\in \cK$
$$
A \sum_{\gamma\in\Gamma} |F(\gamma)|^2 \|K_\gamma\|_\cK^{-1} \le
\|F\|_\cK^2 \le B \sum_{\gamma\in\Gamma} |F(\gamma)|^2 \|K_\gamma\|_\cK^{-1}
\,,
$$
where $K_\gamma$ denotes the reproducing kernel at $\gamma\in\Omega$. 

We now define the sequences for which we establish our sampling
theorem. For $b>0$ let $L_b \subseteq \bbC$ be the square lattice 
\begin{equation}\label{square-lattice-b}
L_b=\Big\{\gamma_{\ell, m}\in\bbC: \gamma_{\ell
  ,m}=\textstyle{\sqrt{\frac{2\pi}{b}}}(\ell+im),
(\ell,m)\in\bbZ^2\Big\}. 
\end{equation}

\begin{defn}\label{samp-seq}{\rm
For $(b_1,\ldots, b_n)\in\bbR_+^n$ and $a\in \bbR_+$, consider
the sequence of 
points $\Gamma \subseteq b\cU$ 
\begin{equation*}
\Gamma =\Big\{  \big( \gamma',
\textstyle{ \frac{\pi}{a}}k+\textstyle{ \frac{i}{4}} |\gamma'|^2
\big)\in \bbC^n\times\bbC:
\ \gamma'\in L_{(b_1,\dots,b_n)} :=L_{b_1}\times\cdots\times L_{b_n},\,
k\in\bbZ \Big\}.
\end{equation*}
}
\end{defn}

 Since  the lattice $\Gamma\subseteq b\cU$,  we can use Heisenberg
coordinates and observe that $\Gamma$ has a product structure;
precisely 
$$\Gamma=\big\{ [\gamma',k\pi/a]\in\bbH_n:\, \gamma' \in
L_{(b_1,\dots,b_n)},\, k\in\bbZ\big\}\, . 
$$ 
\begin{thm}\label{main-5}
Let $\Gamma \subseteq b\cU$ be as in Definition
\ref{samp-seq} and suppose $b_j>a$, $j=1,\ldots, n$. Then, there exist
constants $A_\Gamma ,B_\Gamma>0$ such that for all 
$F\in \PW^n_a$ we have 
$$
A_\Gamma \sum_{\gamma \in\Gamma} |\p_t^{n/2} F(\gamma)|^2 \le \| F\|_{\PW_a^n}^2 
\le B_\Gamma \sum_{\gamma \in\Gamma} |\p_t^{n/2} F(\gamma)|^2  ,
$$
where $t=\Re\z_{n+1}$.  As a consequence, for every $G\in\PW_a$ we
have the sampling 
\begin{equation*}
 A_\Gamma \sum_{\gamma \in\Gamma} |G(\gamma)|^2 
\le  \| G\|_{\PW_a}^2 \le B_\Gamma \sum_{\gamma \in\Gamma} |G(\gamma)|^2 .
\end{equation*}

 Conversely, if there exists $j_0\in\{1,\dots,n\}$ such that
$b_{j_0}\le a$, then the sequence $\Gamma$ fails to be sampling for
$\PW_a$.  
\end{thm}

 We point out that Theorem \ref{main-5} most likely could be generalized to
 more general sequences of points, see the comments in the last
Section \ref{Final-remarks}.

We emphasize here one peculiar difference between the 1-dimensional
and the several variables settings: in the 1-dimensional setting the
Fourier transform of a function in the fractional Paley--Wiener space $PW^s_a$ is supported on a
symmetric interval $[-a,a]$ (\cite[Theorem 1]{MPS-frac}), whereas in
several variables the non-commutative Fourier
transform of a function in $\PW^s_a$ is supported on an interval of the form $[-a,0]$.  
In fact, the upper half-plane
$U=\{\z=x+iy\in\bbC:y>0\}$ and its complement
$U^c=\{\z=x+iy\in\bbC:y\leq 0\}$ have the same geometry. However, our
result is modeled on the Siegel upper-half space $\cU$ and the
geometries of $\cU$ and $\cU^c$ are clearly different. We will see (Lemma \ref{P-L-Siegel}) that a function $F\in \PW^s_a$ is bounded
 in $\cU$, whereas grows exponentially in $\cU^c$.

 The classical Paley--Wiener and Whittaker--Kotelnikov--Shannon
 Theorems have a natural and straightforward extension in several variables at least when we consider entire functions in $\bbC^d$ of exponential type
with respect to the cube $[-a,a]^d$ and the sampling on the lattice $\frac \pi a\bbZ^d$. More generally, it is possible to consider entire functions of exponential type with respect to a symmetric body $K$. A symmetric body
is a convex, compact and symmetric subset of $\bbR^d$ with
non-empty interior. Then, if $f\in \Hol(\bbC^d)$ and $f|_{\bbR^d}\in L^2$, $f$ is of exponential type with respect to $K$ if and only if the Fourier transform of $f|_{\bbR^{d}}$ is supported on
$K$, see \cite[Chapter III]{Stein-Weiss}.
Notice that in this case $\bbR^d$
 is  a totally  real
submanifold  of real co-dimension $d$. Sampling theorems for entire functions in several variables in more general
sets than a dilation of the lattice $\bbZ^d$ have drawn
considerable interest in recent times. In \cite{OU} the authors
discuss the sampling for the Paley--Wiener space of entire functions
in several variables with convex spectrum. In  \cite{GL} the authors
obtain very interesting results concerning some necessary and some
sufficient condition for sampling in Fock spaces in $\bbC^2$ in
connection with  existence of Gabor frames in $\bbR^2$.  In \cite{GHOR}
the authors
prove strict density inequalities for sampling and interpolation in
Fock spaces in $\bbC^d$
defined by a plurisubharmonic weight.
See also \cite{GJM} for related results, and the references in the cited papers.

The paper is organized as follows. In Section \ref{Basic-facts-sec}
we recall some facts about representation theory and 
the fractional Laplacian on the Heisenberg
group. In Section  \ref{sec:3} 
we prove a Plancherel--P\'olya type inequality and 
the Paley--Wiener type result, Theorem \ref{main-4}.  We also
 compute the reproducing
kernels for the spaces $\PW_a^s$, $0\le s<n+1$ and show  in
Section \ref{sec:3} a resemble  between the classical Paley--Wiener space
$PW_a(\bbC)$ in one variable and the fractional space $\PW_a^n$.
In Section \ref{sec:4} we prove a representation theorem for functions
in $\PW^s_a$ that we shall use in the proof of our sampling theorem,
but that we believe is of interest in its own.   Section \ref{sec:5}
is devoted to a careful estimate of the sampling constants for the
Fock space $\cF^\lambda$ when $\lambda
$ varies in the bounded interval $(0,a]$.   In Section
\ref{sec:6} we prove our main result and we conclude with some final remarks and
open questions in Section \ref{Final-remarks}. 
\ms

\section{Preliminaries and basic facts}
\label{Basic-facts-sec}  

\subsection{The Sobolev space $\dot{W}^{s,p}$ and the Fourier transform}
We consider the fractional operator $\Delta^{s/2}$ defined following
\cite{Komatsu, Folland-Arkiv}
as
\begin{equation}\label{Delta-s2-def}
\Delta^{s/2} \vp
= \lim_{\eps\to0}\frac1{\Gamma(k-\frac s2)}\int_\eps^\infty r^{k-\frac s2 -1} e^{-r\Delta}
\Delta^k \vp\, dr \,,
\end{equation}
where $k>s/2$ is an integer, 
whose domain is the set of $\vp\in L^p(\bbH_n)$ for which the limit
exists in $L^p$. Then $\Delta^{s/2}$ is a closed operator on $L^p$,
$1<p<\infty$ and its domain contains the Schwartz space $\cS$, see 
\cite[Thm. (3.15)]{Folland-Arkiv}.  
When $0<s<2n+2$, the operator $\Delta^{s/2}$ has an inverse given by
convolution with a locally integrable homogeneous function.  We denote
such convolution operator by $\cI_s$.  The following result in
contained in \cite[(1.11), (3.18)] {Folland-Arkiv}. Recall that the
homogeneous Sobolev spaces $\dot{W}^{s,p}$ is the completion  of
$\cS$ with respect to the norm $\|\Delta^{s/2}\vp\|_{L^p}$.
\begin{prop}\label{Folland-Riesz-pot}
  Let $1<p<\infty$, $0<s<(2n+2)/p$ and $p^*$ given by $\frac{1}{p^*}=
\frac1p - \frac{s}{2n+2}$.  Then, $\cI_s:L^p\to L^{p^*}$ is bounded.
Therefore, if $1<p<\infty$, $0<s<(2n+2)/p$,
$$
\dot{W}^{s,p} = \big\{ f\in L^{p^*}:\,  \Delta^{s/2} f \in L^p,\, \|
f\|_{\dot{W}^{s,p}}:=\| \Delta^{s/2} f\|_{L^p} \big\}\,.
$$
  \end{prop}
 In particular, we emphasize that if $f\in \dot W^{s,p}$, then there exists a sequence $\{\varphi_k\}_k\subseteq \mathcal S$ such that $\varphi_k\to f$ in $L^{p^*}$ 
 and $\{\Delta^{s/2} \varphi_k\}_{k}$ 
admits a limit in $L^p$. The fractional Laplacian $\Delta^{s/2} f$ of $f$ is set by definition to be such a limit.

Our goal now is to extend the definition of the Fourier transform to
$\dot{W}^{s,2}$ when $0<s<n+1$.  
The differentials of the Bargmann representations, that are
defined in \eqref{Barg-rep} , can be computed 
to give:
\begin{itemize}
\item[(i)] for all $\lambda\neq0$, $d\beta_\lambda(T) = i\lambda$;\smallskip
\item[(ii)] for $\lambda>0$,  $d\beta_\lambda(Z_j) = \p_{w_j} $, $d\beta_\lambda(\ov Z_j) =
  -\frac\lambda2 w_j$;\smallskip  
\item[(iii)]  for $\lambda<0$,  $d\beta_\lambda (Z_j) = \frac\lambda2 w_j$, and 
  $d\beta_\lambda(\ov Z_j) = \p_{w_j} $;
\end{itemize}
see \cite{Folland-phase} or \cite{Thangavelu}.  It is important to recall that, with our choice
of normalization of the Fourier transform, if $f,g\in L^1(\bbH_n)$,
 $\beta_\lambda(f*g) = 
\beta_\lambda(f)\beta_\lambda(g)$, so that for any right-invariant vector field $D$ 
\begin{equation}\label{id-1}
\beta_\lambda(D f)= -d\beta_\lambda(D) \beta_\lambda(f) \,.
\end{equation}
Since $\Delta = -\frac2n \sum_{j=1}^n (Z_j \ov Z_j + \ov Z_j Z_j) $, we
obtain that $d\beta_\lambda(\Delta)$ is the operator on $\cF^\lambda$
such that
$
d\beta_\lambda(\Delta) e_\alpha = |\lambda|(1+|\alpha|/n)e_\alpha$.
In particular, $d\beta_\lambda(\Delta)$ is a diagonal operator on
$\cF^\lambda$ with respect to the
standard basis $\{e_\alpha\}$. 
Therefore, from \eqref{Delta-s2-def} it follows that 
\begin{align}\label{id-2}
  d\beta_\lambda(\Delta^{s/2})  e_\alpha
  & =\frac1{\Gamma(k-\frac s2)}\int_0^\infty r^{k-\frac s2 -1}
  d\beta_\lambda\big( e^{-r\Delta} \Delta^k\big) e_\alpha \, dr
    \notag  \\
   & =\frac1{\Gamma(k-\frac s2)}\int_0^\infty r^{k-\frac s2 -1}
     e^{-r d  \beta_\lambda(\Delta)}
  d   \beta_\lambda(\Delta^k) e_\alpha \, dr \notag\\
  & =\frac1{\Gamma(k-\frac s2)}\int_0^\infty r^{k-\frac s2 -1} e^{-r|\lambda|(1+|\alpha|/n))}
|\lambda|^k(1+|\alpha|/n)^ke_\alpha\, dr \notag\\
&    = \big[
|\lambda|(1+|\alpha|/n)\big]^{s/2} e_\alpha \,.
\end{align}

\begin{lem}\label{F-T-Ws}
Let $0\le s<n+1$.  Then, the Fourier transform on $\bbH_n$ defines a
bounded operator $\beta: \dot{W}^{s,2}\to \cL^2_s$.
\end{lem}
\proof
If $\vp\in\cS$, then $\vp,\Delta^{s/2}\vp \in L^2$, so that both
$\beta_\lambda (\vp),\beta_\lambda(\Delta^{s/2}\vp) \in \HS(\cF^\lambda)$. The
identities \eqref{id-1} and \eqref{id-2} now give that
\begin{equation}\label{F-T-Ws-eq}
\beta_\lambda(\vp)
= |\lambda|^{-s/2} M_\lambda \beta_\lambda(\Delta^{s/2}\vp) ,
\end{equation}
where $M_\lambda$ is the bounded operator on $\cF^\lambda$ such that 
$M_\lambda(e_\alpha) = (1+|\alpha|/n)^{-s/2} e_\alpha$.
Then, 
\begin{align*}
\| \beta(\vp)\|_ {\cL^2_s}^2
  & = \int_{\bbR^*} \|  \beta_\lambda(\vp) \|_{\HS}^2\,
    |\lambda|^{n+s}d\lambda \\
& = \int_{\bbR^*} \| M_\lambda \beta_\lambda(\Delta^{s/2}\vp) \|_{\HS}^2\,
|\lambda|^n d\lambda\\
  & \le  \int_{\bbR^*} \| \beta_\lambda(\Delta^{s/2}\vp) \|_{\HS}^2\,
    |\lambda|^n d\lambda \\
  & = \| \vp\|_{\dot{W}^{s,2}}^2
\,. \qed
\end{align*}

Following \cite{R} and using the above
differentials (i-iii) it is possible to  see how the holomorphicity
forces some constraints on the support of the Fourier 
transform. In particular the following lemma holds, see also \cite[Section
2.3]{AMPS} and \cite{CP}.
\begin{lem}\label{support-holomorphic-fcts}
Let $c\in\bbR$ and set
$$
\cU_c = \Big\{ \z=(\z',\z_{n+1}):
\Im\z_{n+1}>\textstyle{\frac{1}{4}}|\z'|^2+c\Big\}.
$$
Let $F\in\Hol(\cU_c)$,  
$\widetilde F=F\circ\Psi^{-1}$ and set $\widetilde F_h[z,t]=\widetilde
F(z,t,h)$. If $h>c$ and $\widetilde F_h\in L^2(\bbH_n)$, then
$\beta_\lambda  
(\widetilde F_h)=0$ for $\lambda>0$ and $\ran \big( \beta_\lambda  
(\widetilde F_h)\big) \subseteq \operatorname{span}\{e_0\}$.
\end{lem}

\begin{remark}\label{remark-OV}{\rm
    As a consequence of the lemma, if 
$$
\cV = \Big\{ \z\in\bbC^{n+1}:\,
\textstyle{\frac{|\z'|^2}{4}}+c-\delta<\Im \z_{n+1}<
\textstyle{\frac{|\z'|^2}{4}}+c+\delta \Big\}
$$
is a neighborhood of $b\cU_c$, $F\in\Hol(\cV)$ and $\widetilde
F_t\in L^2(\bbH_n)$ for $t\in(c-\delta,c+\delta)$, then
\begin{equation*}
\sigma_{\lambda}(\widetilde F_{c+h})=e^{\lambda h}\sigma_{\lambda}(\widetilde F_c)
\end{equation*}
for any $h\in(-\delta,\delta)$, see e.g. \cite{AMPS,MM,OV}.}
\end{remark}

We now recall that  
the operators $i^{-1}T$ and $\Delta$ admit commuting self-adjoint
extensions on $L^2(\bbH_n)$, see \cite{Strichartz91} or
 \cite[Ch. 2]{Thangavelu-book}.
If $F\in\Hol(\cV)$, where $\cV$ is a tubular neighborhood
of $b\cU$, and $\widetilde F_0\in L^2(\bbH_n)$,  then $(\Delta+iT)\widetilde F_0=0$, that is,
$\Delta\widetilde F_0= -iT\widetilde F_0$.  Therefore,
\begin{equation}\label{spec-id-1}
  \Delta^{s/2}\widetilde F_0= (-iT)^{s/2}\widetilde F_0
\end{equation}
for all such
$F$'s and $s>0$. 
Hence,
\begin{equation}\label{spec-id-2}
 \beta_\lambda(\Delta^{s/2} \widetilde F_0) = \beta_\lambda((-iT)^{s/2}
 \widetilde F_0)=|\lambda|^{s/2}\beta_\lambda(\widetilde F_0)\,.  
\end{equation}

\subsection{Fock spaces, lattices and Weierstrass $\sigma$-functions}
We now
recall some facts on lattices and the associated Weierstrass
$\sigma$-functions.

For  $b>0$, we let $L_b$ to be the square lattice 
\begin{equation}\label{square-lattice-lambda}
L_b=\big\{ 
 \gamma_{\ell m}\in\bbC:\, \gamma_{\ell m} = \textstyle{
   \sqrt{\frac{2\pi}{b}} }(\ell+im),\, (\ell,m)\in\bbZ^2\big\},
\end{equation}
For such a lattice $L_b$ we consider the
Weierstrass $\sigma$-function associated to $L_b$, 
\begin{equation*}
\sigma_{L_b} (z) = z \prod_{(\ell,m)\in\bbZ^2\setminus(0,0)}
\bigg( 1-\frac{z}{\gamma_{\ell m}}\bigg) \exp\bigg\{ 
  \frac{z}{\gamma_{\ell m}} + \frac{z^2}{2\gamma_{\ell m}^2} \bigg\}.
\end{equation*}

 We recall a few well-known properties of the  the
Weierstrass $\sigma$-function $
\sigma_{L_b}$ for any $b>0$, see e.g. \cite[Ch.1]{Zhu}:
\begin{itemize}
\item[(i)]  $\sigma_{L_b}$ is an entire function of order $2$ and   type $\frac b4$  that vanishes exactly
at the points of $L_b$;
\item[(ii)] for all $z\in\bbC$, $|\sigma_{L_b}(z)|e^{-\frac
    b4 |z|^2}$ is   
  double periodic with periods $\sqrt{2\pi/b}$ and $i\sqrt{2\pi/b}$
 and it is
  bounded above and below   by  constants $C_b, c_b$ resp., depending
  only on $b$ 
   times $d(z,L_b)$, the euclidean distance of $z$ from the lattice $L_b$,  for
  all $z\in\bbC$;  
\item[(iii)] there exists a constant $c_b'>0$ depending only
  on $b$, such that  
  for all $\gamma_{\ell m}\in L_b$,
\begin{equation}\label{sigma'-estimate}
|\sigma_{L_b}'(\gamma_{\ell m})|e^{-\frac b4 |\gamma_{\ell
      m}|^2} \ge c_b'.
    \end{equation}
\end{itemize}

We also recall that given the lattice $L_b$, then for any
any $f\in \cF^{b'}$ with $b'<b$ we have the decomposition
\begin{equation}\label{decom-sigma-fct:eq}
f(z) = \sum_{\gamma_{\ell m}\in L_b} \frac{f(\gamma_{\ell m})}{\sigma_{L_b}'
  (\gamma_{\ell m}) } \frac{\sigma_{L_b} (z)}{z-\gamma_{\ell m}} 
\end{equation}
 where the series converges in $\Hol(\bbC)$, see
 e.g. \cite[Proposition 4.24]{Zhu}.
\ms

\section{The Plancherel--P\'olya inequality}\label{sec:3}
In this section we prove our first results. We begin with
a Plancherel--P\'olya type inequality adapted to the Siegel half-space.
This result implies in particular that the spaces $\PW^s_a$ are
complete, for $0\le s<n+1$.

\subsection{The Plancherel--P\'olya inequality}
We now prove a Phrang\'em--Lindel\"of type
result for the Siegel half-space. We first need the following modified version of the classical result in the complex plane. 

\begin{lem}\label{P-L-plane}
Let $g\in \Hol(\bbC)$ and suppose that there exist constants $c,a, M>0$ such that:
\begin{itemize}
 \item[(i)] $ |g(t)| \le M$,
 \item[(ii)] for every $\eps>0$ there exists $C_\eps>0$ such that 
$$
|g(w)| \le C_\eps e^{(a+\eps)( c|t|+ |u|)},$$
where $w=t+iu$.
\end{itemize}
Then,
$$
|g(w)| \le M e^{a|u|} \,. 
$$
\end{lem}

The classical proof applies also here and we skip the details; see,
for instance, \cite{Young}.  In the Siegel half-space we have the following variation.

\begin{lem}\label{P-L-Siegel}
Let $F\in\Hol(\bbC^{n+1})$ and suppose that there exists constant $c,a,M>0$ such that:
\begin{itemize}
 \item[(i)]$ |F_{|_{b\cU}}(\z)| =
|\widetilde F_0[z,t] | \le  M$,
\item[(ii)] for every $\eps>0$ there exists $C_\eps>0$ such that 
$$|F(\z)|=|\widetilde F(z,t,h)| \le C_\eps e^{(a+\eps)( c|[z,t]|^2 +|h|)}.
$$
\end{itemize}
Then, setting $h_- =-\min(0,h)$, we have
$$
|F(\z)|=|\widetilde F(z,t,h)| \le Me^{ah_-}\,.
$$
\end{lem}

\proof
For $w=t+iu\in\bbC$ and for every fixed $\z'\in\bbC^n$ we define
$$
g_{\z'}(w) = F(\z',w+\textstyle{\frac i4}|\z'|^2)=\widetilde F(z,t,u).
$$
Then, $g_{\z'}$ is entire and from $(ii)$ above we get 
\begin{align*}
|g_{\z'}(w)|
& \le C_\eps e^{(a+\eps)( c|[z,t]|^2+ |u|)} \le C_\eps e^{(a+\eps)(
  c\frac{|z|^2}{4}+c|t|+ |u|)} \\
& \le C'_\eps(\z') e^{(a+\eps)( c|t|+ |u|)}
\end{align*}
where the constant $C'_\eps(\z')$ depends on the fixed $\z'\in\bbC^n$. Moreover,
$$
|g_{\z'}(t)|  = |F(\z',t+\textstyle{\frac i4}|\z'|^2)| \le M,
$$
where $M$ is an absolute constant not depending on $\z'$. Lemma \ref{P-L-plane} implies
$$
|g_{\z'}(w)|\le M e^{a|u|} \,.
$$
Thus, setting $w=
\z_{n+1}-\textstyle{\frac i4}|\z'|^2$ we have $h=\rho(\z)= u$ and
$$
|F(\z',\z_{n+1})| = |F(\z', w+\textstyle{\frac i4}|\z'|^2)|= |g_{\z'}(w)| \le M e^{a|h|} \,. 
$$
In order to complete the proof, we need to show that we can improve
the above inequality when $h>0$, by showing that in fact $|F(\z)|\le
M$ when $\z\in\ov\cU$.  For each $\z_{n+1}$ fixed we have
$$
\sup_{\z':\, \frac{|\z'|^2}{4}\le\Im\z_{n+1} } |F(\z',\z_{n+1})| 
=  \sup_{\z':\, \frac{|\z'|^2}{4}=\Im\z_{n+1} } |F(\z',\z_{n+1})| \le 
\sup_{\z\in b\cU} |F(\z',\z_{n+1})|\,.
$$
Therefore, 
\begin{align*}
\sup_{\z\in\ov\cU} |F(\z)|
& = \sup_{(\z',\z_{n+1}):\, \frac{|\z'|^2}{4}\le\Im\z_{n+1} } |F(\z)|
\le \sup_{\z\in b\cU} |F(\z',\z_{n+1})| \le M\,,
\end{align*}
as we wished to show. 
\epf

From this Phragm\'en--Lindel\"of principle we deduce a
version of the
Plancherel--P\'olya inequality in this setting. For $c\in\bbR$ we set $\cU_c = \big\{
\z=(\z',\z_{n+1}): 
\Im\z_{n+1}>\frac{|\z'|^2}{4}+c\big\}$.  

\begin{prop}{\bf (Plancherel--P\'olya Inequality)}\label{P-P-prop}
Let $F\in\cE_a$ be such that $\widetilde F_0\in L^p(\bbH_n),
1<p<\infty$.  Then, for all $h\in\bbR$,
$$
\int_{\bbH_n} |\widetilde F_h[z,t]|^p\, dzdt \le 
e^{aph_-} \|\widetilde
F_0\|^p_{L^p(\bbH_n)} \,,
$$
where $(z,t,h)=\Psi(\z)$ and $h_- =-\min(0,h)$. In particular, $F\in H^p(\cU_c)$, the Hardy space on $\cU_c$, for all $c\in\bbR$. 
\end{prop}
\proof
Let $\vp\in C^\infty_c(\bbH_n)$, $0\le\vp\le1$, $\|\vp\|_{L^{p'}(\bbH_n)}
\le 1$, where $\textstyle{\frac 1p+\frac 1{p'}}=1$, and define
\begin{align*}
G(\z)
& =  \int_{\bbH_n} \widetilde  F_h\big( [z,t][w,s]\big) \vp[w,s]\, dwds  \,.
\end{align*}
Then, if $\z\in b\cU$, i.e. $\Psi(\z)=(z,t,0)\equiv[z,t]\in \bbH_n$, 
$$
|G(\z)| \le \|\widetilde  F_0\|_{L^p(\bbH_n)} \|\vp\|_{L^{p'}(\bbH_n)}
\le \|\widetilde  F_0\|_{L^p(\bbH_n)} \,.
$$
Moreover, for $\z\in\bbC^{n+1}$, using  \cite{Cygan} (with a
slight abuse of notation)  we have 
\begin{align*}
 \| \big([z,t][w,s], h\big) \|_\cU  
& = \big| [z,t][w,s]\big|^2 +|h|  \le \big( | [z,t]|+|[w,s]|\big)^2
+|h| \\
& \le 2\big( [z,t]|^2+|[w,s]|^2\big) +|h| \,.   
\end{align*}
  Therefore,
\begin{align}
| G (\z) |
& \le C_\eps 
 \int_{\bbH_n} e^{(a+\eps)( 2[z,t]|^2+|h| +2|[w,s]|^2)}  |\vp[w,s]|\,
 dwds  \notag \\
& \le C'_\eps  e^{(a+\eps)( 2|[z,t]|^2+|h|)} \,, \label{extra-exp-type-est}
\end{align}
since $\vp$ has compact support.  By Lemma \ref{P-L-Siegel} we obtain
$$
|G(\z)| \le  e^{ah_-}  \|\widetilde F_0\|_{L^p(\bbH_n)}\,,
$$
that is,
$$
\Big| 
\int_{\bbH_n} \widetilde F_h \big([z,t][w,s]\big)\vp[w,s]\, dwds \Big|
\le  e^{ah_-}  \|\widetilde F_0\|_{L^p(\bbH_n)} \,,
$$
for every $\vp\in C^\infty_c(\bbH_n)$, $\| \vp\|_{L^{p'}(\bbH_n)}\le
1$.  
Therefore,
$$
\|\widetilde F_h \|_{L^p(\bbH_n)} 
= \|\widetilde F_h ([\cdot,\cdot][w,s]) \|_{L^p(\bbH_n)}  \le 
e^{ah_-} \|\widetilde F_0\|_{L^p(\bbH_n)} \,.
$$
The conclusions follow.
\epf

\proof[Proof of Theorem \ref{main-4}]
We begin with the case $s=0$. 
 Let $F\in \PW_a$.  By
Proposition \ref{P-P-prop} it follows that $\| \widetilde F_h\|_{L^2(\bbH_n)}
\le e^{ah_-} \|\widetilde F_0\|_{L^2(\bbH_n)} $ and $F\in H^2(\cU_c)$ for
all $c\in\bbR$. In particular, arguing as in Remark \ref{remark-OV} we
obtain that 
$$
\beta_\lambda(\widetilde F_h) = e^{\lambda h}
\beta_\lambda(\widetilde F_0)  
$$
 for all $h\in\bbR$. Thus, thanks to the Paley--Wiener characterization of $H^2$, we have $\sigma_{\lambda}(\widetilde F_0)=0$ for $\lambda>0$. By Plancherel's formula it then follows that
\begin{align*}
\| \widetilde F_h \|_{L^2(\bbH_n)}^2
& = \int_{-\infty}^0 \| \beta_\lambda(\widetilde F_h)\|_{\HS}^2
|\lambda|^n\, d\lambda = \int_{-\infty}^0 e^{2\lambda h}\| \beta_\lambda(\widetilde F_0)\|_{\HS}^2
|\lambda|^n\, d\lambda \,,
\end{align*}
whereas, by Proposition \ref{P-P-prop} 
\begin{align*}
\| \widetilde F_h \|_{L^2(\bbH_n)}^2
& \le e^{2ah_-} 
 \int_{-\infty}^0\| \beta_\lambda(\widetilde F_0)\|_{\HS}^2
|\lambda|^n\, d\lambda \,.
\end{align*}
Therefore, for all $h\in\bbR$,
$$
\int_{-\infty}^0 e^{2\lambda h}\| \beta_\lambda(\widetilde F_0)\|_{\HS}^2
|\lambda|^n\, d\lambda 
\le e^{2ah_-} 
 \int_{-\infty}^0 \| \beta_\lambda(\widetilde F_0)\|_{\HS}^2
|\lambda|^n\, d\lambda\,,
$$
and, by letting $h\to-\infty$,
this easily implies that $\supp \beta_\lambda(\widetilde F_0)
\subseteq [-a,0)$.

To prove the converse direction, given $\tau\in \cH^2 \big([-a,0)\big)$, arguing as \cite[Lemma 3.1]{AMPS}, we see that for
every $\lambda<0$, 
$
\big| \tr
\big(\tau(\lambda) \beta_\lambda[z,t]^* \big)
\big| \le \| \tau(\lambda)\|_{\HS} $.  
Therefore,
\begin{align*}
 \int_{-a}^0 e^{\lambda h} 
\big| \tr
\big(\tau(\lambda) \beta_\lambda[z,t]^* \big)
\big|  |\lambda|^n\, d\lambda 
&  \le \int_{-a}^0 e^{\lambda h} \| \tau(\lambda)\|_{\HS} |\lambda|^n\,
 d\lambda\notag \\
& \le \| \tau\|_{\cL^2} \bigg( 
\int_{-a}^0 e^{2\lambda h}  |\lambda|^{n}\,
 d\lambda \bigg)^{1/2} \notag\\
& \le C \| \tau\|_{\cL^2}  e^{ah_-} \,,
\end{align*}
This shows that the integral in \eqref{PW-D-eq3} converges
absolutely. Let $F$ be given by \eqref{PW-D-eq3}. The same
argument as in \cite[Lemma 3.1]{AMPS} now shows that $F$ is entire,
hence in $\cE_a$ by the previous estimate.
Moreover, $F$  is such that  
$\beta_\lambda(\widetilde F_0) =\tau(\lambda)\in \cL^2$, so that by
Plancherel's formula $\widetilde F_0\in L^2(\bbH_n)$, that is, $F\in \PW_a$.
\ms

We now consider the case $s>0$.  If $F\in \PW_a^s$, then by Lemma \ref{F-T-Ws}
$\beta_\lambda(\widetilde{F_0})$ is well defined and
$\beta(\widetilde{F_0})\in \cL^2_s$. 
Let  $\vp$
be a Schwartz function on $\bbH_n$ such that $\beta_\lambda(\vp)=
\vp_0(\lambda) \la \cdot, e_0\ra e_0$ for $\vp_0\in C^\infty$ having
support in $[-N,-1/N]$, for some $N>0$.
 
We define $\widetilde G_h [z,t] = (\widetilde F_h *\vp) [z,t]$,
so that 
$$
G(\z)
=  \int_{\bbH_n} \widetilde F_h\big([z,t][w,s]^{-1}\big)\vp[w,s]\, dwds \,.
$$
We claim that $G\in \cE_a$.  
Indeed, observe that $\beta_\lambda(\widetilde G_0) = 
 \beta_\lambda(\widetilde F_0) \beta_\lambda(\vp) =
 \vp_0(\lambda)\beta_\lambda(\widetilde F_0)$, has compact
 support contained in $[-N,-1/N]$.
 Since $\beta(\widetilde F_0) \in \cL^2_s$, it follows that
 $\beta(\widetilde G_0) \in \cL^2$, which in turns gives $\widetilde
 G_0\in L^2(\bbH_n)$. 
The first part of the theorem  now shows
that $G$ is entire function of exponential type at most $N$.  However,
since $F\in \cE_a$, arguing as in \eqref{extra-exp-type-est}  we
also have 
\begin{align*}
|G(\z)| 
& \le  \int_{\bbH_n} |\widetilde F_h\big([z,t][w,s]^{-1}\big) |\, 
|\vp [w,s]|\, dwds \\
& \le C  e^{(a+\eps)( 2|[z,t]|^2+|h|)}  \,.
\end{align*}
Since $\widetilde G_0\in L^2(\bbH_n)$, by the previous case $s=0$,
$G\in \PW_a$ and 
$\supp(\beta_\lambda(\widetilde G_0)) \subseteq [-a,0)$.  By the
choice of $\vp$, this easily implies that also 
$\supp(\beta_\lambda(\widetilde F_0)) \subseteq [-a,0)$,
and that $\ran \beta_\lambda(\widetilde F_0)\subseteq
\operatorname{span}\{e_0\}$.

By Lemma \ref{F-T-Ws}, in particular by \eqref{F-T-Ws-eq}, we have
$$
\beta_\lambda \big(\Delta^{s/2} \widetilde F_0 \big) e_\alpha =
|\lambda|^{s/2} [1+|\alpha|/n]^{s/2} \beta_\lambda(\widetilde F_0) e_\alpha 
\,.
$$

Then,
\begin{align*}
\| \beta_\lambda \big(\Delta^{s/2} \widetilde F_0 \big)\|_{\HS}^2 & =
\sum_\alpha \big| \la \beta_\lambda \big(\Delta^{s/2} \widetilde F_0 \big)
e_\alpha , e_\alpha\ra_{\cF^\lambda} \big|^2 = |\lambda|^s \| \beta_\lambda  (\widetilde F_0 )\|_{\HS}^2 \,.
\end{align*}
Hence, 
\begin{align*}
\| F\|_{\PW_a^s}^2 
& = \int_{\bbH_n} |\Delta^{s/2}\widetilde F_0 [z,t]|^2\, dzdt \\
& = \int_{-\infty}^0 \|\beta_\lambda(\Delta^{s/2}\widetilde
F_0) \|_{\HS}^2
  |\lambda|^n\, d\lambda \\
& = \int_{-\infty}^0 \|\beta_\lambda(\widetilde F_0) \|_{\HS}^2
  |\lambda|^{n+s}\, d\lambda \,.
\end{align*}
 
In particular, $\beta_\lambda(\widetilde F_0)\in \cH^2_s$, with equality
of norms.  
 
Conversely, let 
$\tau\in \cH_s^2 \big([-a,0)\big)$ and $F$ be given by
\eqref{PW-D-eq3}. We have that
\begin{align}
 \int_{-a}^0 e^{\lambda h} 
\big| \tr
\big(\tau(\lambda) \beta_\lambda[z,t]^* \big)
\big|  |\lambda|^n\, d\lambda 
&  \le \int_{-a}^0 e^{\lambda h} \| \tau(\lambda)\|_{\HS} |\lambda|^n\,
 d\lambda \notag \\
& \le \| \tau\|_{\cL^2_s} \bigg( 
\int_{-a}^0 e^{2\lambda h}  |\lambda|^{n-s}\,
 d\lambda \bigg)^{1/2} \notag\\
& \le C \| \tau\|_{\cL^2_s}  e^{ah_-} \,, \label{estimate-sollevamento-2}
\end{align}
where $C<+\infty$ if and only if $s<n+1$. In this case we 
can conclude that $F\in
\cE_a$. 
Now, we claim that $\beta_\lambda (\Delta^{s/2} \widetilde F_0) = |\lambda|^{s/2} 
\tau(\lambda)$, so that 
\begin{align*}
\| \Delta^{s/2} \widetilde F_0\|_{L^2(\bbH_n)}^2
& = \int_{-a}^0 \| \beta_\lambda (\Delta^{s/2} \widetilde F_0)\|_{\HS}^2
|\lambda|^n \, d\lambda\\
& = \int_{-a}^0 \| \tau(\lambda)\|_{\HS}^2
|\lambda|^{n+s} \, d\lambda \\
& = \| \tau\|_{\cL^2_s}^2 \,,
\end{align*}
as we wished to show.  It remains to prove the claim. It is easy
to construct fields of operators $\eta_\eps$ such that
$\eta_\eps \in \cL^2(-a+\eps,-\eps)$ be smooth in $\lambda$ and 
$\eta_\eps\to \tau$ in $\cH^2_s (-a,0)$ as $\eps\to0$. Then, the
function 
$$
G_\eps(\z)=\widetilde
G_{\eps,h}[z,t]:=\frac{1}{(2\pi)^{n+1}}
\int_{-a}^{-\eps}e^{\lambda h}\tr\big(\eta_\eps(\lambda) \beta_\lambda[z,t]^*\big)\,
|\lambda|^n d\lambda
$$
is in $\cS(\bbH_n)$. 
Hence, using \eqref{spec-id-2} and \eqref{inve-form},  we have that
$$
\beta_\lambda(\Delta^{s/2}(\widetilde
G_{\eps,0}))= |\lambda|^{s/2} \beta_\lambda(\widetilde
G_{\eps,0}) =
|\lambda|^{s/2}\eta_\eps(\lambda) .
$$
Since $\eta_\eps\to \tau$ in $\cH^2_s (-a,0)$,  $|\lambda|^{s/2}\eta_\eps 
\to |\lambda|^{s/2}\tau$  in  $\cH^2(-a,0)$, so  that
$\Delta^{s/2}(\widetilde G_{\eps, 0})$ converges in
$L^2(\bbH_n)$ to a function $G$ such that $\beta_\lambda
(G)=|\lambda|^{s/2}\tau(\lambda)$. Moreover, since $\dot W^{2,s}$
embeds continuously in $L^{2^*}$, we also have that $\widetilde
G_{\eps, 0}$ is a Cauchy  in $L^{2^*}$ and its limit
 is $\widetilde F_0$. Then, by definition, $\Delta^{s/2}\widetilde F_0= G$ and the claim follows.  \epf

As a consequence of the Paley--Wiener theorems we obtain that the space 
$\PW^s_a, 0\leq s<n+1$, is a reproducing kernel Hilbert space and we explicitly compute its kernel.
We set
$$
Q(\omega,\z)=
\textstyle{\frac{1}{2i}}  (\omega_{n+1}-\ov\z_{n+1} ) - \textstyle{\frac14}
  \omega'\cdot\ov\z'\,,
$$
so that, by  writing
$\zeta=\big(z, t+i(h+\frac{|z|^2}{4})\big)$, 
$\omega= \big(w, u+i(k+\frac{|w|^2}{4})\big)$, 
$$
\widetilde Q(z,t,h;w,u,k) = \textstyle{\frac{1}{2i}} \Big(
u-t+\textstyle{\frac12} \Im(w\cdot\bar z ) +i  \big( h+k+\textstyle{\frac14} |w-z|^2 \big)\Big)\,.
$$

\begin{cor}\label{thm-kernel}
 Let $s\in [0,n+1)$. Then, the space $\PW^s_a$ is a reproducing kernel Hilbert space with reproducing kernel
 \begin{equation}\label{PW-kernel}
  K(\omega,\z)=K_{\z}(\omega)= \frac{1}{(2\pi)^{n+1}} \int_{-a}^{0} 
 e^{2i\lambda Q(\omega,\z)  }
|\lambda|^{n-s}\,
d\lambda \, ,
\end{equation}
 and  $\beta_\lambda\big( \widetilde{K_{\z,0}}\big) =
\chi_{[-a,0)}(\lambda) e^{\lambda
  h}|\lambda|^{-s}P_0\beta_\lambda[z,t]$, where $P_0$ denotes the orthogonal projection onto the subspace
generated by $e_0$. 
\end{cor}

\begin{proof}
  The Plancherel--P\'olya Inequality,  Proposition \ref{P-P-prop},
  implies that $\PW_a^s$  continuously  embeds into
  $\Hol(\bbC^{n+1})$.  Hence,
 the completeness of $\PW^s_a$ and the boundedness of the point-evaluation functionals follow.
 
 The explicit computation of the kernel follows from a standard
 argument.   Let $\tau$ denote the element of $\cH^2_s([-a,0))$ and define 
\begin{align*}
F(\z)&= \widetilde F_h[z,t]=
\frac{1}{(2\pi)^{n+1}} \int_{-a}^0
e^{\lambda  h}
\tr\big(\tau(\lambda) \beta_\lambda[z,t]^*\big)\,
|\lambda|^n d\lambda\\
&=
\frac{1}{(2\pi)^{n+1}} \int_{-a}^0
e^{\lambda  h}
\tr\big(\tau(\lambda) P_0\beta_\lambda[z,t]^*\big)\,
|\lambda|^n d\lambda.
\end{align*}
The last identity holds since
$\ran\tau_F(\lambda)\subseteq\operatorname{span}\{e_0\}$. We also have 
\begin{align*}
  F(\z)
  &= \widetilde F_h[z,t]  =\langle F, K_\z\rangle_{\PW^s_a}=
 \frac{1}{(2\pi)^{n+1}}\int_{-a}^0
\tr\big( \tau(\lambda)\beta_\lambda( (\widetilde K_\z)_0)^*\big)
\, |\lambda|^{n+s}\, d\lambda.
\end{align*}

Since the above identities hold for all $\tau\in \cH^2_s([-a,0))$ it follows that 
$$
\beta_\lambda( (\widetilde K_\z)_0)=e^{\lambda h} \chi_{[-a,0)}(\lambda) |\lambda|^{-s}P_0\beta_\lambda[z,t].
$$
From the inversion formula \eqref{PW-D-eq3},
and arguing as in the
proof of \cite[Corollary 4.3]{AMPS} to compute $\tr\big(P_0
\beta_\lambda[z,t]  \beta_\lambda[w,s]^*\big)$, 
 we obtain that
\begin{align*}
K_\z (\omega) 
& = \frac{1}{(2\pi)^{n+1}} \int_{-a}^0
e^{\lambda  (h+k)}
\tr\big(P_0 \beta_\lambda[z,t]  \beta_\lambda[w,s]^*\big)\,
 |\lambda|^{n-s} d\lambda \\
& = \frac{1}{(2\pi)^{n+1}} \int_{-a}^0
e^{\lambda  (h+k +\frac14 |w-z|^2 +i (t-s -\frac12 \Im w\cdot\bar z))}
\,  |\lambda|^{n-s} d\lambda 
\end{align*}
and the conclusion  follows.
\end{proof}

\begin{remark}{\rm  In particular, in the case  $s=n$ the
    reproducing kernel 
$K(\omega,\z)$ of $\PW_a^n$ takes a more familiar expression, that
involves the $\sinc$ function, $\sinc z=\frac{\sin z}{z}$. Namely, 
$$
K(\omega,\zeta) = (2\pi)^{-n-1} a e^{-i a Q(\omega,\z)} \sinc \big(a Q(\omega,\z)\big) \,. 
$$
We also observe that
\begin{align*}
\| K_\z\|_{\PW_a^s}^2 & = \frac{1}{(2\pi)^{n+1}} 
\int_{-a}^0 e^{2h\lambda} \|\, |\lambda|^{-s}P_0\beta_\lambda[z,t] \|_{\HS}^2 \,
|\lambda|^{n+s} d\lambda = \frac{1}{(2\pi)^{n+1}} 
                        \int_{-a}^0   e^{2h\lambda} \, |\lambda|^{n-s} d\lambda  \\
  & \approx_{s,a} e^{2ah_-}\,.
\end{align*}
In particular, $\| K_\z\|_{\PW_a^s}\approx_{s,a}1$ for $\z\in b\cU$.}
\end{remark}

\section{A representation theorem for $\PW^s_a$}\label{sec:4}
In this section we  prove a representation theorem for functions in
$\PW^s_a$, for $s\in[0,n+1)$.
 We denote by $\sF$ the $1$-dimensional Euclidean Fourier transform, that is, for $f\in L^1(\bbR)$,
 $$
 \sF f(\xi)=\int_{\bbR} f(x)  e^{ix\xi} \, dx.
 $$ 
 We also write $\sF f =\widehat f$.
 Recall that $\sF$ extends to a surjective isomorphism $\sF: L^2(\bbR)\to  L^2(\bbR)$ where
$$ 
\|f\|_{L^2(\bbR)}^2=\frac{1}{2\pi}\|\sF f\|_{L^2(\bbR)}^2  
$$
and the inverse $\sF^{-1}$ is defined as
$$
\sF^{-1}f(x)=\frac 1{2\pi}\int_\bbR f(\xi)  e^{-i\xi x}  \, d\xi.
$$ 
We recall that if $f\in L^2$ and
$\widehat f$ has compact support, then $f$
extends to an entire function $F$ and
\begin{equation}\label{F-T-formula}
\sF\big(F(\cdot+iy)\big)(\lambda) =
\widehat{f}(\lambda)e^{y\lambda}.
\end{equation}

\begin{thm}\label{key2} 
Let $F\in \PW_a^s$, $0\le s<n+1$.   For $\z'\in\bbC^n$ fixed, define
$f_{\z'} (\kappa) = F(\z',\kappa)$, where $\kappa =x+iy\in\bbC$. Then
there exists $\phi:\bbC^n\times\bbR\to \bbC$ such that the function
$
\sF^{-1}  \phi(\z',\cdot)  (x)
$
extends to an entire function in the variable $\kappa$ and it holds that
\begin{equation}\label{identity}
F(\z',\z_{n+1}) =
\sF^{-1}  \phi(\z',\cdot)  (\z_{n+1}) \,,
\end{equation}
Moreover, the function $\phi$ satisfies the following:
\begin{itemize}
\item[(i)] $\phi(\cdot,\lambda) \in \cF^\lambda$ for a.e. $\lambda\in
  [-a,0)$;\smallskip 
\item[(ii)]  $\|\phi(\cdot,\lambda) \|_{\cF^\lambda} \in L^2 \big(
  [-a,0), |\lambda|^{s-n}d\lambda\big)$; \smallskip
\item[(iii)] 
  $ \displaystyle{ \| F\|_{\PW_a^s}^2
= (2\pi)^{n-1} \int_{-a}^0\|\phi(\cdot,\lambda)\|^2_{\cF^\lambda}\,
|\lambda|^{s-n}\, d\lambda}$.
\end{itemize}
If $s=n$ we also have
\begin{itemize}
 \item[(iv)] $\phi(\z',\cdot) \in
   L^2([-a,0))$ for all $\z'\in\bbC^n$; in particular $F(\z',\cdot)$
   belongs to the one-dimensional Paley--Wiener space $PW_a(\bbC)$ for all $\z'\in\bbC^{n}$. 
\end{itemize}
\end{thm}

\proof
Observe that from Theorem \ref{main-4} it follows that
$\PW_a\cap \PW_a^s$ is dense in both spaces. Then, if $F\in \PW_a\cap
\PW_a^s$, 
the computations that follow are all justified.
By Theorem \ref{main-4} we have 
\begin{align*}
  F(\z',\z_{n+1}) & = \widetilde F_h [z,t]
 = \frac{1}{(2\pi)^{n+1}} \int_{-a}^0  e^{\lambda h} \tr \big( 
\beta_\lambda (\widetilde F_0) \beta_\lambda [z,t]^* \big)
|\lambda|^n\, d\lambda \\
& = \frac{1}{2\pi} \int_{-a}^0 \Big( \frac{|\lambda|}{2\pi} \Big)^n
\tr \big(  \beta_\lambda (\widetilde F_0) \beta_\lambda [z,0]^* \big)
e^{-i\lambda(t+ih)} \, d\lambda \,,\\
&= \frac{1}{2\pi} \int_{-a}^0 \Big( \frac{|\lambda|}{2\pi} \Big)^n
\tr \big(  \beta_\lambda (\widetilde F_0) \beta_\lambda [\z',0]^* \big)
e^{-\frac\lambda 4|\z'|^2}e^{-i\lambda\z_{n+1}} \, d\lambda  \,.
\end{align*} 
Therefore, setting
\begin{equation*}
\phi(\z',\lambda)  =   \Big(
\frac{|\lambda|}{2\pi} \Big)^n  \tr \big(  \beta_\lambda (\widetilde F_0)
\beta_\lambda [\z',0]^* \big) e^{-\frac\lambda 4|\z'|^2}\chi_{[-a,0)}(\lambda)
\,,
\end{equation*}
it follows that
$
F(\z',\z_{n+1}) =
\sF^{-1} \phi(\z',\cdot) (\z_{n+1}  ) $, 
that is, \eqref{identity} holds.
From \eqref{Barg-rep} we deduce that $\phi(\cdot,\lambda)$ is entire,
and by 
\eqref{identity}, using \eqref{F-T-formula} and \eqref{spec-id-1}, 
it follows that
\begin{align*}
\| F\|_{\PW_a^s}^2
& = \int_{\bbH_n} |\Delta^{s/2} \widetilde F_0 [z,t]|^2 \, dzdt \\
& = \int_{\bbH_n} |\, |T|^{s/2} \widetilde F_0 [z,t]|^2 \, dzdt \\
& = \int_{\bbC^n} \int_\bbR |\, |T|^{s/2} F (z,t+\textstyle{\frac i4} |z|^2)|^2 
\, dt\, dz \\
& = \frac{1}{2\pi} \int_{\bbC^n} \int_\bbR 
 |\sF\big( |T|^{s/2} F (z,\cdot +\textstyle{\frac i4} |z|^2) \big)(\lambda)|^2 
\, d\lambda \, dz \\
& = \frac{1}{2\pi} \int_{-a}^0 |\lambda|^s \int_{\bbC^n}\big|  \phi(z,\lambda)\big|^2 e^{-\frac{|\lambda|}{2}|z|^2 }\, dz \,
  d\lambda  \\
&=(2\pi)^{n-1} \int_{-a}^0\|\phi(\cdot,\lambda)\|^2_{\cF^\lambda}\, |\lambda|^{s-n}\, d\lambda.
\end{align*} 
The conclusions (i-iii) now follow. About (iv), if $P_0$ denotes the orthogonal projection onto the subspace generated by $e_0$, we have
\begin{align*}
|  \phi(\z',\lambda)|^2&=|\big( |\lambda|/(2\pi) \big)^n\tr \big(  \beta_\lambda (\widetilde F_0)
\beta_\lambda [\z',0]^* \big) e^{-\frac\lambda 4|\z'|^2}\chi_{[-a,0)}(\lambda)|^2\\
&=|\big( |\lambda|/(2\pi) \big)^n\tr \big(  \beta_\lambda (\widetilde F_0)
P_0\beta_\lambda [\z',0]^* \big) e^{-\frac\lambda 4|\z'|^2}\chi_{[-a,0)}(\lambda)|^2\\
&\leq \|\beta_\lambda(\widetilde F_0)\|^2_{\HS}\|P_0\beta_\lambda[\z',0]\|^2_{\HS}e^{-\frac\lambda 2|\z'|^2}|\big( |\lambda|/(2\pi) \big)^n\chi_{[-a,0)}(\lambda)|^2\\
&\leq e^{\frac a 2|\z'|^2} \|\beta_\lambda(\widetilde F_0)\|^2_{\HS}|\big( |\lambda|/(2\pi) \big)^n\chi_{[-a,0)}(\lambda)|^2
\end{align*}
where we used $\lambda<0$ and the identity $\|P_0\beta_\lambda [\z',0]\|^2_{\HS}=1$. Since $F\in \PW^n_a$ 
$$
\int_{-a}^0\|\beta_\lambda(\widetilde F_0)\|^2_{\HS}
\big(|\lambda|/(2\pi)\big)^{2n}\, d\lambda  <\infty, 
$$
and this completes the proof.
\epf

As a consequence, we have the following.   For $0\le s< n+1$ we set
\begin{multline*}
\Upsilon_s
= \Big\{ \phi: \bbC^n\times[-a,0)\to\bbC:\ (i)\,
\phi(\cdot,\lambda)\in\cF^\lambda \ \text{for a.e. } \lambda\in
[-a,0) ,\\
(ii)\, \|\phi\|^2_\Upsilon := (2\pi)^{n-1} \int_{-a}^0 \| \phi(\cdot,\lambda)\|_{\cF^\lambda}^2
|\lambda|^{s-n}\, d\lambda<\infty \Big\} \,.
\end{multline*}
\begin{cor}\label{unit-map-cor}
  For $0\le s<n+1$, 
  the mapping
  $$
\sU:  \PW^s_a\ni F\mapsto \sF \big( F_{|_{\Im \z_{n+1}=0}}\big) \in \Upsilon_s
  $$
  is a unitary map  (where $\sF$ denotes the Euclidean Fourier
  transform in the real part of the variable $\z_{n+1}$); in particular
  $$
  \|F\|_{\PW^s_a}^2
 =(2\pi)^{n-1}\int_{-a}^0\|\phi(\cdot,\lambda)\|^2_{\cF^\lambda}\, |\lambda|^{s-n}\, d\lambda.
  $$
\end{cor}
\proof
We only need to prove that the mapping is onto.  Given
$\phi\in\Upsilon_s$, setting
$F(\z',\z_{n+1}) =\sF^{-1}\phi(\z',\cdot)(\z_{n+1}) $
it is easy to see that $F\in\PW^s_a$ and the conclusion follows.
\qed
\ms

\section{Sampling in the Fock space}\label{sec:5}

In this section we prove a result, Theorem \ref{new-lem},  that it may be considered as {\em
  folklore}. We consider the $1$-dimensional case and 
make explicit the dependence on $\lambda$ of the sampling
constant in the case of square lattices for the Fock space $\cF^\lambda(\bbC)$.  However, we believe that
the result is not completely obvious and it is key for our Theorem
\ref{main-5}.    We recall that a square lattice $L_b$ in $\bbC$ is
sampling for $\cF^a(\bbC)$ if and only if $b>a$
(\cite{Seip,Seip-Wallsten}) and that the behavior of the sampling constants
as $b\to a^+$ are obtained in \cite{BGL}.

\begin{lem}\label{fock-b'}
 Let $a>0$ be given, let $b>a$ and let $L_b$ be the square lattice
 \eqref{square-lattice-b}.
 Let $f\in \mathcal F^\lambda$ with $0<\lambda\leq a$. Then, for any $\eta\in L_b$, the function
 $$
 F^\lambda_\eta(z):= e^{\frac\lambda2 z\overline\eta} f(z-\eta)
 $$
 belongs to the Fock space $\mathcal F^{b'}$ with $a<b'<b$.
\end{lem}
\proof
We have
\begin{align*}
\int_\bbC |F^\lambda_\eta(z)|^2e^{-\frac {b'}2|z|^2}\, dz
&= \int_\bbC |e^{\frac\lambda 2(z+\eta)\overline{\eta}}f(z)|^2 e^{-\frac{b'}2|z+\eta|^2}\, dz\\
&=  e^{\frac\lambda 2|\eta|^2}\int_{\bbC} |f(z)|^2 e^{-\frac{b'-\lambda}{2}|z+\eta|^2}e^{-\frac\lambda 2|z|^2}\, dz\\
&\leq e^{\frac\lambda 2|\eta|^2}\int_\bbC |f(z)|^2 e^{-\frac\lambda 2|z|^2}\, dz
\end{align*}
and the conclusion follows.
\qed
\begin{lem}\label{schur}
Let $a>0$ be given, let $b>a$ and let $L_b$ be the square lattice \eqref{square-lattice-b}. For $0<\lambda\leq a$ define the positive measure
$$
\mu^\lambda_{L_b}:=\sum_{\gamma\in L_b}e^{-\frac\lambda 2|\gamma|^2}\delta_{\gamma}
$$
where $\delta_\gamma$ is the unit point measure at $\gamma$ and consider the integral operator
$$
f\mapsto Tf=\int_\bbC K_t(\cdot,\eta)f(\eta)\, d\mu^\lambda_{L_b}(\eta)
$$
with positive kernel 
$$
K_t(\gamma,\eta)=   e^{\frac\lambda4 (|\gamma|^2+|\eta|^2)}
e^{-\frac{t-\lambda}{4}|\gamma+\eta|^2} , \qquad a<t<b.
$$
Then, the operator $T$ extends to a bounded operator $T:L^2(L_b,
\mu^\lambda_{L_b})\to L^2(L_b, \mu^\lambda_{L_b})$ with operator norm
uniformly bounded for $0<\lambda\leq a$. 
\end{lem}
\proof
For $\gamma\in L_b$ we have
$$
Tf(\gamma)= \sum_{\eta\in L_b} K_t(\gamma,\eta)f(\eta) e^{-\frac\lambda 2 |\eta|^2}.
$$
By Schur's test \cite[Appendix A.2]{Grafakos} it is enough to find $\varphi>0$ and $C>0$ such that 
$$
\sum_{\eta\in L_b} K_t(\gamma,\eta)\varphi(\eta) e^{-\frac\lambda 2 |\eta|^2}\leq C\varphi(\gamma).
$$
This would also guarantee that the operator norm of $T$ is bounded by $C$. Choosing $\varphi(\gamma)= e^{\frac\lambda 4|\gamma|^2}$ the conclusion follows with a constant $C$ independent of $\lambda$ as we wished to show.
\qed

\begin{thm}\label{new-lem}
  Let $a>0$ be given, let $b>a$ and let $L_b$ be the square lattice
 \eqref{square-lattice-b}. 
 Then there exist constants $A,B>0$ 
such that for all $0<\lambda\le a$ and all $f\in\cF^\lambda(\bbC)$ we have
$$
A \lambda \sum_{\gamma\in L_b} |f(\gamma)|^2
e^{-\frac\lambda2 |\gamma|^2} \le \| f\|_{\cF^\lambda}^2 \le  B \lambda \sum_{\gamma\in L_b} |f(\gamma)|^2
e^{-\frac\lambda2 |\gamma|^2} .
$$
\end{thm}

\proof
Let $0<\lambda\leq a$ and let $f\in \cF^\lambda$ be given. Let $R_{L_b}$ be the fundamental region of the square lattice $L_b$ and let $R_{L_b, \eta}$ be the translated region $R_{L_b}+\eta$ where $\eta\in L_b$. Then, $\bbC=\bigcup_{\eta\in L_b}R_{L_\lambda,\eta}$ and the intersections
of the $R_{L_\lambda,\eta}$'s have Lebesgue measure zero. Therefore,
\begin{align}
 \| f\|_{\cF^\lambda}^2
& = \frac{\lambda}{2\pi} 
\sum_{\eta\in L_b} \int_{R_{L_b,\eta}} |f(z)|^2e^{-\frac{\lambda}{2}|z|^2}\, dz \notag\\
&=\frac{\lambda}{2\pi}\sum_{\eta\in L_b}\int_{R_{L_b}}|f(z-\eta)|^2e^{-\frac\lambda 2|z-\eta|^2}\, dz \notag\\
&=\frac{\lambda}{2\pi}\sum_{\eta\in L_b} e^{-\frac\lambda 2|\eta|^2}\int_{R_{L_b}}|e^{\frac\lambda 2 z\overline{\eta}}f(z-\eta)|^2 e^{-\frac\lambda 2|z|^2}\, dz.
\end{align}
In particular the factor $e^{-\frac \lambda 2 |z|^2}$ is bounded above and below on the region $R_{L_b}$ with positive constants uniformly on $z$ and $\lambda$. Hence, we conclude that 
\begin{equation}
\|f\|^2_{\cF^\lambda}\approx\lambda\sum_{\eta\in L_b} e^{-\frac\lambda 2|\eta|^2}\int_{R_{L_b}}|e^{\frac\lambda 2 z\overline{\eta}}f(z-\eta)|^2\, dz,
\end{equation}
that is, the two quantities are comparable up to some positive constants which do not depend on $\lambda$.

Now, setting $F^\lambda_{\eta}(z)= e^{\frac\lambda 2z\overline{\eta}}f(z-\eta)$, Lemma \ref{fock-b'} guarantees the decomposition
$$
F^\lambda_{\eta}(z)=\sum_{\gamma\in L_b}\frac{F^\lambda_{\eta}(\gamma)}{\sigma'_{L_b}(\gamma)}\frac{\sigma_{L_b}(z)}{z-\gamma}.
$$

Since
$\Big|\frac{\sigma_{L_b}(z)}{z-\gamma}\Big|\leq C$
for a constant $C$ which does not depend on $z\in  R_{L_b}$ and $\gamma\in L_b$, we have
\begin{align*}
  \|f\|^2_{\cF^\lambda}
  &  \le C 
  \lambda\sum_{\eta\in L_b} e^{-\frac\lambda 2|\eta|^2}\int_{R_{L_b}}|e^{\frac\lambda 2 z\overline{\eta}}f(z-\eta)|^2\, dz\\
 &\leq C \lambda \sum_{\eta\in L_b} e^{-\frac\lambda 2 |\eta|^2}\int_{R_{L_b}}\bigg(\sum_{\gamma\in L_b}\Big|\frac{F^\lambda_\eta(\gamma)}{\sigma'_{L_b}(\gamma)}\Big|\bigg)^2\, dz\\
 &\leq C \frac{2\pi \lambda}{b}\sum_{\eta\in L_b} e^{-\frac\lambda2|\eta|^2}\bigg(\sum_{\gamma\in L_b}\big|F^\lambda_\eta(\gamma)e^{-\frac t4|\gamma|^2}\big|\bigg)^2
\end{align*}
where $a<t<b$ and we used the estimate \eqref{sigma'-estimate}. Now,
\begin{align*}
 \frac{2\pi\lambda}{b}\sum_{\eta\in L_b} e^{-\frac\lambda2|\eta|^2}\bigg(\sum_{\gamma\in L_b}\big|F^\lambda_\eta(\gamma)e^{-\frac t4|\gamma|^2}\big|\bigg)^2&=\sum_{\eta\in L_b} e^{-\frac\lambda2|\eta|^2}\bigg(\sum_{\gamma\in L_b}\big|e^{\frac\lambda 2|\eta|^2+\frac\lambda2\gamma\overline\eta-\frac t4|\gamma+\eta|^2}f(\gamma)\big|\bigg)^2\\
 &=\frac{2\pi\lambda}{b}\sum_{\eta\in L_b}e^{-\frac\lambda 2|\eta|^2}\bigg(\sum_{\gamma\in L_b}K_t(\gamma,\eta)|f(\gamma)|e^{-\frac\lambda 2|\gamma|^2}\bigg)^2,
\end{align*}
where we have set 
$$
K_t(\gamma,\eta)= e^{\frac\lambda4
  (|\gamma|^2+|\eta|^2))}e^{-\frac{t-\lambda}{4}|\gamma+\eta|^2}. \
$$
 Thus, from Lemma \ref{schur} we get
$$
\sum_{\eta\in L_b}e^{-\frac\lambda 2|\eta|^2}\bigg(\sum_{\gamma\in L_b}K_t(\gamma,\eta)|f(\gamma)|e^{-\frac\lambda 2|\gamma|^2}\bigg)^2\leq C\sum_{\eta\in L_b}|f(\eta)|^2e^{-\frac\lambda 2|\eta|^2}
$$
where $C$ does not depend on $\lambda$. In conclusion, we have
\begin{align*}
\|f\|^2_{\mathcal F^\lambda}&\leq B\lambda \sum_{\eta\in L_b}e^{-\frac\lambda 2|\eta|^2}\bigg(\sum_{\gamma\in L_b}\big|F^\lambda_\eta(\gamma) e^{-\frac t4|\gamma|^2}\big|\bigg)^2\leq B\lambda \sum_{\eta\in L_b}|f(\eta)|^2 e^{-\frac\lambda 2|\eta|^2}
\end{align*}
with $B$ independent of $\lambda$ as we wished to show.

Next, denoting by $D(\gamma,r)$ the
disk centered at $\gamma\in\bbC$ with radius $r>0$,  we show that for
all $f\in\Hol(\bbC)$ and $d>0$ we  have
\begin{equation}\label{mean-value-ineq}
  |f(\gamma)|^2 e^{-d|\gamma|^2}
  \le 
  \frac{d}{\pi\big(1-e^{-dr^2}\big)}
  \int_{D(\gamma,r)} |f(w)|^2 e^{-d|w|^2}\, dw \,.
\end{equation}
For, by the mean value formula we have that
$$
|f(\gamma)|^2 \le \frac{1}{2\pi} \int_0^{2\pi}
|f(\gamma+re^{i\theta})|^2\,d\theta \,,
$$
so that,
\begin{align*}
 \int_{D(\gamma,r)} |f(w)|^2 e^{-d|w|^2}\, dw
& = \int_{D(0,r)} |f(w+\gamma)|^2 e^{-d|w+\gamma|^2}\, dw\\
& = \int_{D(0,r)} \big|f(w+\gamma) e^{-dw\overline{\gamma}} \big|^2 e^{-d(|w|^2+|\gamma|^2)}\, dw\\
  & = \int_0^r  e^{-d(s^2+|\gamma|^2)}
\int_0^{2\pi}
    \big|f( se^{i\theta}+\gamma) e^{-dse^{i\theta}\overline{\gamma}} \big|^2
    \,d\theta sds\\
  & \ge 2\pi  |f(\gamma)|^2 e^{-d|\gamma|^2}
    \int_0^r e^{-ds^2} \, sds\\
  & = \frac{\pi}{d}\big(1-e^{-dr^2} \big) |f(\gamma)|^2 e^{-d|\gamma|^2}
  ,
\end{align*}
and \eqref{mean-value-ineq} follows.  Finally, given the lattice
$L_b$, we let $0<r<\inf\{|\gamma_1-\gamma_2|:\, \gamma_1,\gamma_2\in L_b\}$.
Then, the disks $\{D(\gamma,r):\, \gamma\in L_b\}$ are disjoint so
that, by \eqref{mean-value-ineq} we have
\begin{align*}
\lambda \sum_{\gamma\in L_b} |f(\gamma)|^2
  e^{-\frac\lambda2 |\gamma|^2}
  & \le \frac{\lambda}{\big(1-e^{-\frac\lambda2 r^2}\big)}
\sum_{\gamma\in L_b} \frac{\lambda}{2\pi}
  \int_{D(\gamma,r)} |f(w)|^2 e^{-\frac\lambda2 |w|^2}\, dw \\
& \le   \sup_{\lambda\in(0,a]} \frac{\lambda}{\big(1-e^{-
      \frac\lambda2 r^2}\big)}  \| f\|_{\cF^\lambda}^2, 
\end{align*}
and the conclusion follows with $A= \bigg( \sup_{\lambda\in(0,a]} \frac{\lambda}{\big(1-e^{-
      \frac\lambda2 r^2}\big)}  \bigg)^{-1}$.
\epf

The following result now follows easily.
\begin{cor}\label{n-dim-version}
 For $j=1,\ldots,n$ let $a>0$ be given, let $b_j>a$ and set
 $L_{(b_1,\dots,b_n)} =L_{b_1}\times\cdots\times L_{b_n}$.  Then, there exist constants $A,B>0$
such that for all $\lambda\in (0,a]$  and $f\in\cF^\lambda(\bbC^n)$ we have
$$
A \lambda ^n \sum_{\gamma\in L_{(b_1,\dots,b_n)}} |f(\gamma)|^2
e^{-\frac{\lambda}{2}{|\gamma|^2}} \le \| f \|_{\cF^\lambda}^2
\le B \lambda ^n \sum_{\gamma\in L_{(b_1,\dots,b_n)}} |f(\gamma)|^2
e^{-\frac{\lambda}{2}{|\gamma|^2}} \,.
$$
\end{cor}

 We remark that the main point in Theorem \ref{new-lem} and Corollary \ref{n-dim-version}
is to estimate the dependence of the constants as $\lambda$ approaches
$0$.  In fact, if we restrict the parameter $\lambda$ to vary in a
compact interval $[\delta,a]$ for some $0<\delta<a<b$, then we can
replace the square lattice $L_b$ by a general sampling sequence for
$\cF^a$.  In~\cite{Seip} and~\cite{Seip-Wallsten} it was shown that a sequence $Z\subseteq\bbC$ is a sampling
sequence for $\cF^a$ if and only if $Z$ 
is the union of finitely many separated sequences and $Z$
contains a separated subsequence $Z'$ such that $D_-(Z') > a/(2\pi)$; see
also~\cite{Zhu}.  
Here
$$
D_- (Z') = \lim_{r\to\infty} \inf_{w\in\bbC}
\frac{\#\big(D(w,r)\cap Z'\big)}{\pi r^2}\,,
$$
and $Z'$ is said to be separated if
$$
\inf_{w,w'\in Z'} |w-w'|\ge c_{Z'}>0 \,.
$$

In fact, the following result holds.
\begin{prop}\label{referee-prop}
  Let $0<\delta<a$ be given.  
  For $j=1,\ldots,n$  let $Z_j$ be a sampling sequence for
  $\cF^a(\bbC)$ and let $Z=Z_{(1,\dots,n)}:=Z_1\times\cdots\times Z_n$ and let
 $b:=2\pi\min_{j=1,\dots,n}D_-(Z_j)>a$.
 Then, there exist constants $A,B>0$ depending only on $\delta,b$ and
 the separation constants of $Z_j$, $j=1,\dots,n$, 
such that for all $\lambda\in [\delta,a]$  and $f\in\cF^\lambda(\bbC^n)$ we have
$$
  A\lambda ^n \sum_{\gamma\in Z} |f(\gamma)|^2
e^{-\frac{\lambda}{2}{|\gamma|^2}} \le \| f \|_{\cF^\lambda}^2
\le B \lambda ^n \sum_{\gamma\in Z} |f(\gamma)|^2
e^{-\frac{\lambda}{2}{|\gamma|^2}}\,.
$$
\end{prop}
\proof
This follows from the arguments in \cite{Seip,Seip-Wallsten} (see
also~\cite[Chapter4]{Zhu}).
Notice that we may assume that $n=1$ and then that $Z$ is separated.  Then,
$Z$ is a sampling sequence for $\cF^\lambda(\bbC)$, for
any $\lambda\in[\delta, a]$ with constants $A',B'>0$ that depend only on the
separation constant of $Z$ and on $D_-(Z)$ such that
$$
 A'\sum_{\gamma\in Z} |f(\gamma)|^2
e^{-\frac{\lambda}{2}{|\gamma|^2}} \le \| f \|_{\cF^\lambda}^2
\le B' \sum_{\gamma\in Z} |f(\gamma)|^2
e^{-\frac{\lambda}{2}{|\gamma|^2}}\,.
$$
The conclusion now follows since we are assuming $\lambda\in[\delta,a]$. 
\epf
\ms

 \section{Sampling in $\PW_a$}\label{sec:6}
Before proving Theorem \ref{main-5}, we study a few properties of
$\PW_a^n$.  In particular we present some elements  and produce an explicit orthonormal basis
of such
space. We also remark that because of the Fourier transform
characterization of $\PW_a^n$ the Fock spaces $\mathcal F^\lambda$
that will appear in this section are defined for negative $\lambda$ in
$[-a,0)$ and that, by definition, $\mathcal F^\lambda=\mathcal
F^{|\lambda|}$.  

We use both the notation $\sF g$ and $\widehat g$ to denote the
1-dimensional Euclidean Fourier transform of $g\in L^2(\bbR)$.  
Let $g\in
L^2(\bbR)$ such that $\supp
\widehat g \subseteq
[-a,0]$, we set $G(\z',\z_{n+1})=  g (\z_{n+1})$, where we denote by
$g$ its entire extension to $\bbC$  
(notice that $G$ is independent of $\z'\in\bbC^n$).  Then we compute
\begin{align*}
  \| G\|_{\PW^n_a}^2
  & = \int_{\bbH_n} \big|\Delta^{n/2} \widetilde{G_0} [z,t]\big|^2\,
    dzdt
  =  \int_{\bbC^n} \int_\bbR \big|\p_t^{n/2}
    G\big(z,t+{\textstyle{ \frac{i}{4}| }} \z'|^2\big) \big|^2\,
    dzdt \\
  &= \frac{1}{2\pi} \int_{\bbC^n} \int_{-a}^0  |\lambda|^n  \big|   (\sF G)
    ( z,\lambda) \big|^2 e^{\frac\lambda2 |z|^2}\,
   d\lambda dz = \frac{1}{2\pi} \int_{\bbC^n} \int_{-a}^0 |\lambda|^n  |   \widehat g 
    (\lambda) |^2 e^{\frac\lambda2 |z|^2}\,
    d\lambda dz \\
  &= (2\pi)^{n-1} \int_{-a}^0 |  \widehat  g 
    (\lambda) |^2\,
    d\lambda \,.
\end{align*}
Hence, 
$G\in\PW^n_a$  and $ \| G\|_{\PW^n_a}^2=  (2\pi)^n \| g \|_{L^2(\bbR)}^2$.
More generally, given a multiindex $\alpha$, we set
\begin{equation*}
  G_\alpha(\z',\z_{n+1}) =
  \frac{1}{\sqrt{2^{|\alpha|} \alpha!}}  (\z')^\alpha
  \p_t^{|\alpha|/2}g  (\z_{n+1}) \,.
\end{equation*}

\begin{lem}\label{G-alpha-lem}
  The following properties hold.
  \begin{itemize}
  \item[(i)] For every
    $\alpha$ we have
    $$
\| G_\alpha \|_{\PW^n_a}^2 =(2\pi)^n \| g\|_{L^2(\bbR)}^2 = (2\pi)^n \sum_{k\in\bbZ}
\big| g\big( {\textstyle{\frac{\pi}{a}}} k\big) \big|^2 .
    $$
  \item[(ii)] Let $g_\ell\in L^2(\bbR)$ be such that $\supp(\widehat
    g_\ell)\subseteq[-a,0]$, 
    $\{\widehat g_\ell:\, \ell\in\bbZ \}$ is an orthonormal basis of
    $L^2(-a,0)$ and set
    $$
    G_{\alpha,\ell} (\z',\z_{n+1})=  \frac{1}{\sqrt{2^{|\alpha|} \alpha!}}  (\z')^\alpha
    \p_t^{|\alpha|/2}g_\ell  (\z_{n+1}).
    $$
    Then  $\big\{
  G_{\alpha,\ell}:\, \alpha\in\bbN^n,\, \ell\in\bbZ\big\}$ is an
  orthonormal basis of $\PW^n_a$.  
    \end{itemize}
  \end{lem}
  \proof
  We observe that for any $r>0$, $(\p_t)^r g =\sF^{-1} \big(
  |\lambda|^r \widehat g\big) \in PW_a$.  Then,
 by \eqref{F-T-formula} again, 
  \begin{align*}
    \| G_\alpha \|_{\PW^n_a}^2 
    & 
  =  \int_{\bbC^n} \int_\bbR \big|\p_t^{(n+|\alpha|)/2}
    G_\alpha \big(z,t+{\textstyle{ \frac{i}{4} }}  |z|^2\big) \big|^2\,
      dzdt \\
  &=  \frac{1}{2\pi} \frac{1}{2^{|\alpha|} \alpha!}  
    \int_{\bbC^n} | z^\alpha|^2\int_{-a}^0  |\lambda|^{n+|\alpha|} |\widehat
    g (\lambda)|^2  e^{\frac\lambda2 |z|^2}\,
    d\lambda dz \\
    & =  (2\pi)^{n-1} \int_{-a}^0   |\widehat
    g (\lambda)|^2 \,
    d\lambda = (2\pi)^n \| g\|_{L^2(\bbR)}^2 \,.
 \end{align*}
Conclusion (i)  now follows from the classical
Whittaker--Kotelnikov--Shannon theorem.
In order to prove  (ii) we argue in a similar fashion:
\begin{align*}
  \la G_\alpha,  \, G_\beta\ra_{\PW^n_a}
  & = \int_{\bbH_n} \Delta^{n/2}
   \widetilde{(G_\alpha)_0} [z,t]\overline{\Delta^{n/2}
   \widetilde{(G_\beta)_0}[z,t]} \, dzdt  \\
&  =  \int_{\bbC^n} \int_\bbR \big|\p_t^{(n+|\alpha|)/2}
    G_\alpha \big(z,t+{\textstyle{ \frac{i}{4} }} |z|^2\big)
\overline{ \p_t^{(n+|\beta|)/2} G_\beta \big(z,t+{\textstyle{ \frac{i}{4} }} |z|^2\big)}
\,      dzdt \\
  & = \frac{1}{\sqrt{2^{|\alpha|+|\beta|} \alpha! \beta!}}
    \int_{\bbC^n} z^\alpha \overline{z^\beta} \int_\bbR
    \p_t^{(n+|\alpha|)/2} g_\ell\big(t+{\textstyle{ \frac{i}{4}}}
    |z|^2\big)
    \overline{\p_t^{(n+|\beta|)/2} g_m\big(t+{\textstyle{ \frac{i}{4}}} |z|^2\big) } \, dtdz \\
  &=  \frac{1}{2\pi} \frac{1}{\sqrt{2^{|\alpha|+|\beta|} \alpha! \beta!}}
    \int_{\bbC^n}  z^\alpha\overline{z^\beta} \int_{-a}^0  |\lambda|^{n+(|\alpha|+|\beta|)/2} \widehat
    g_\ell (\lambda) \overline{\widehat
    g_m (\lambda)}   e^{\frac\lambda2 |z|^2}\,
    d\lambda dz \\
    & =  \delta_{\alpha,\beta}  \la g_\ell,\, g_m\ra_{L^2(\bbR)} \,.
\end{align*}
Thus, $\{G_{\alpha,\ell}\}$ is an orthonormal system.  We show that it
also complete.  Let $F\in\PW^n_a$ be orthogonal to
$\{G_{\alpha,\ell}:\, \alpha\in\bbN^n, k\in\bbZ\}$.  Using the same
computation as above we see that
\begin{align*}
  \la F,\, G_{\alpha,\ell}\ra_{\PW^n_a}
  & = \frac{1}{2\pi} \frac{1}{\sqrt{2^{|\alpha|} \alpha!}}
\int_{\bbC^n} \overline{z^\alpha} \int_{-a}^0 |\lambda|^{n+|\alpha|/2} (\sF F)(z,\lambda) \overline{\widehat
    g_\ell(\lambda)} e^{\frac\lambda2|z|^2}\, d\lambda dz\\
& =  \frac{1}{2\pi}\frac{1}{\sqrt{2^{|\alpha|} \alpha!}} \int_{-a}^0 |\lambda|^{n+|\alpha|/2} 
\int_{\bbC^n}  (\sF F)(z,\lambda)
\overline{z^\alpha}e^{\frac\lambda2|z|^2}
\, dz\, \overline{\widehat
    g_\ell(\lambda)} \, d\lambda \\
& = (2\pi)^{n-1}  \int_{-a}^0 (\sF F)_{\alpha,\lambda} (\lambda)
                                      \overline{\widehat g_\ell(\lambda)} \, d\lambda ,
\end{align*}
where we denote by  $ (\sF F)_{\alpha,\lambda} (\lambda)  $ the
Fourier coefficient of $\sF F(\cdot,\lambda)$ in $\cF^\lambda$
w.r.t. the basis $\{e_{\alpha,\lambda}\}$, that is, $
e_{\alpha,\lambda} = z^\alpha/\|z^\alpha\|_{\cF^\lambda}$.
Since $F$  is orthogonal to $\{G_{\alpha,\ell}\}$ for all
$\ell\in\bbZ$, it follows that  $(\sF F)_{\alpha,\lambda} (\lambda)=0$
$\lambda$-a.e., and then by Proposition \ref {unit-map-cor} that $F=0$. \epf

We now prove a necessary condition for certain sequences in
$\bbC^{n+1}$.  The sequence we consider are more general than the ones
in Definition \ref{samp-seq}, but, again in Heisenberg coordinates, are still 
Cartesian product of sequences in $\bbC^n$ and $\bbR$, resp.
Precisely, for $a>0$ and a separated sequence $\cZ'$ in $\bbC^n$ given, let
\begin{equation}\label{prod-seq}
\cZ = \big\{\gamma\in\bbC^{n+1}:\,
\gamma=( \gamma', \textstyle{ \frac{\pi}{a}k +\frac i4 } |\gamma'|^2):\, 
  \gamma\in\cZ',\, k\in\bbZ \big\}. 
\end{equation}

\begin{thm}\label{nec-cond-thm}
Let $a>0$ be fixed and let $\cZ$ be as in \eqref{prod-seq}. If $\cZ$ is
a sampling sequence for $\PW_a$, then $\cZ'$ is a sampling sequence
for $\cF^a(\bbC^n)$.
\end{thm}

\proof
Let $\eps>0$, $f\in \cF^{a-\eps}(\bbC^n)$ and $g_\eps\in L^2(\bbR)$,
$\|g_\eps\|_{L^2}=1$, and such that $\supp(\widehat{g_\varepsilon})\subseteq[-a,-a+\eps]$.   Note that
$f\in\cF^\lambda$ for all $\lambda\ge a-\eps$, and that
$\lambda\mapsto \|f\|_{\cF^\lambda}$ is continuous in
$[a-\eps,\infty)$. Also, $g_\varepsilon$ is in $PW_a(\mathbb C)$.

Let 
$$
F_\eps(\z',\z_{n+1}):= f(\z')g_\eps(\z_{n+1}) .
$$
Using Plancherel theorem  and \eqref{F-T-formula},  we compute
\begin{align*}
\|F_\eps\|_{\PW_a}^2
& = \int_{\bbH_n} \big|\p_t^{n/2} F_\eps(z, t+\frac i4 |z|^2)\big|^2\, dzdt \\
& =   \frac{1}{2\pi}  \int_{\bbC^n} |f(z)|^2 \int_\bbR |\lambda|^n
 |\widehat{g_\eps}(\lambda)|^2 e^{\frac\lambda2 |z|^2}\, d\lambda dz \\
& =   \frac{1}{2\pi}\int_{-a}^{-a+\eps}
 |\widehat{g_\eps}(\lambda)|^2  |\lambda|^n \int_{\bbC^n} |f(z)|^2 
e^{-\frac{|\lambda|}{2} |z|^2}\, dzd\lambda \\
& = (2\pi)^{n-1} \int_{-a}^{-a+\eps} |\widehat{g_\eps}(\lambda)|^2 
\| f\|_{\cF^{|\lambda|}}^2\, d\lambda.
\end{align*}
Now, in addition, suppose also that $|\widehat{g_\eps}(\lambda)|^2 \to
\delta_a$ in the vague topology as $\eps\to0^+$ (e.g. we may take $\widehat{g_\eps} =
\eps^{-1/2}\chi_{[-a,-a+\eps]}$). Then, the right hand side above
tends to $ (2\pi)^{n-1}\| f\|_{\cF^a}^2$ as $\eps\to0^+$ (since $\lambda\mapsto
\| f\|_{\cF^{|\lambda|}}^2$ is continuous). It follows that
\begin{equation}\label{limit-cond}
\lim_{\eps\to0^+} \|F_\eps\|_{\PW_a}^2 = (2\pi )^{n-1}\| f\|_{\cF^a}^2 .
\end{equation}

Suppose now $\cZ$ as in \eqref{prod-seq}  is a sampling sequence for
$\PW_a$ with sampling constants $A,B>0$. Then, by also applying Whittaker--Kotelnikov--Shannon theorem to $g_\varepsilon$, we have
\begin{align*}
\|F_\eps\|_{\PW_a}^2
& \le B \sum_{\gamma\in\cZ} |F_\eps(\gamma)|^2 
= B \sum_{\gamma'\in\cZ'} |f(\gamma')|^2 \sum_{k\in\bbZ} |g_\eps
( {\textstyle \frac{\pi}{a}k +\frac i4 |\gamma'|^2}) |^2 \\
& =  B \sum_{\gamma'\in\cZ'} |f(\gamma')|^2 \int_\bbR 
 |g_\eps( {\textstyle t +\frac i4 |\gamma'|^2}) |^2 \, dt \\
  & =  \frac{B}{2\pi} \sum_{\gamma'\in\cZ'} |f(\gamma')|^2
\int_{-a}^{-a+\eps} |\widehat{g_\eps}(\lambda)|^2 
e^{\frac\lambda2 |\gamma'|^2}\, d\lambda .
\end{align*}
Therefore, using the fact that $\|\widehat{g_\eps}\|_{L^2}=1$ and that $\cZ'$ is separated, we
obtain 
$$
 \|F_\eps\|_{\PW_a}^2
\le \frac{B}{2\pi} \sum_{\gamma'\in\cZ'} |f(\gamma')|^2 
e^{-\frac{a-\eps}{2} |\gamma'|^2}\leq C\|f\|^2_{\mathcal F^{a-\varepsilon}},
$$
where the last inequality follows from the fact that $\cZ'$ is separated.
Analogously,
\begin{align*}
\sum_{\gamma'\in\cZ'} |f(\gamma')|^2 
e^{-\frac a2 |\gamma'|^2}  & \le \sum_{\gamma'\in\cZ'} |f(\gamma')|^2
\int_{-a}^{-a+\eps} |\widehat{g_\eps}(\lambda)|^2 
e^{\frac\lambda2 |\gamma'|^2}\, d\lambda\\
 & = 2\pi \sum_{\gamma\in\cZ} |F_\eps(\gamma)|^2 \\
&  \le \frac{2\pi}{A}   \|F_\eps\|_{\PW_a}^2 .
\end{align*}

Letting $\eps\to0^+$ and using \eqref{limit-cond} we have
\begin{equation}\label{sam-cond}
\frac{A}{(2\pi)^n} \sum_{\gamma'\in\cZ'} |f(\gamma')|^2 
e^{-\frac a2 |\gamma'|^2} \le \|f\|_{\cF^a}^2
\le \frac{B}{(2\pi)^n}  \sum_{\gamma'\in\cZ'} |f(\gamma')|^2 
e^{-\frac a2 |\gamma'|^2}
\end{equation}
for all $f\in \cF^{a-\eps}(\bbC^n)$.  We now claim that
$\bigcup_{\eps>0}\cF^{a-\eps}(\bbC^n)$ is contained and dense in
$\cF^a(\bbC^n)$.  Once the  claim is proven, from \eqref{sam-cond} it follows
that $\cZ'$ is a sampling sequence for 
$\cF^a(\bbC^n)$ and the desired conclusion will follow.  We
observe that:
\begin{itemize}
\item[{\tiny $\bullet$}] for $0<r<1$, $\|
  f(\sqrt{r}\cdot)\|_{\cF^{ar}}= \| f\|_{\cF^a}$ and $\|f(\sqrt{r}\cdot)\|_{\cF^a} = \| f\|_{\cF^{a/r}}\le \| f\|_{\cF^a}^2$;
\item[{\tiny $\bullet$}] $\lim_{r\to1^-} \|f(\sqrt{r}\cdot)\|_{\cF^a}
  = \| f\|_{\cF^a}$;
  \item[{\tiny $\bullet$}] $ f(\sqrt{r}\cdot) \to f$ pointwise as
    $r\to1^-$. 
\end{itemize}
Notice that the last two conditions imply that $f(\sqrt{r}\cdot) \to
f$ also in $\cF^a$-norm. The claim is then proven, so is the theorem.
\epf
 
\proof[Proof of Theorem \ref{main-5}]
For $F\in \PW_a^n$, let $\phi(\z',\lambda)=\sF F(\z', \lambda)$. By 
Theorem \ref{key2} 
\begin{align*}
\| F\|_{\PW_a^n}^2
& = (2\pi)^{n-1} \int_{-a}^0 \| \phi(\cdot,\lambda)\|_{\cF^\lambda}^2 
\, d\lambda 
\,.
\end{align*}
Given the sequence of points $\Gamma$ as in the statement, we write $\Gamma'=L_{(b_1,\dots,b_n)}$.
 By Corollary \ref{n-dim-version} and denoting by $A,B>0$ the constants therein, we have
\begin{align*}
\| F\|_{\PW_a^n}^2
& \le B (2\pi)^{n-1} \int_{-a}^0 |\lambda|^n\sum_{\gamma'\in \Gamma'} |\phi(\gamma',\lambda)|^2
e^{-\frac{|\lambda|}{2}{|\gamma'|^2}} 
\, d\lambda \\
&  =  B (2\pi)^{n-1} \sum_{\gamma'\in \Gamma'} 
\int_{-a}^0   |\lambda|^n |\phi(\gamma',\lambda)|^2
e^{-\frac{|\lambda|}{2}{|\gamma'|^2}} \, d\lambda \\
  &  =B (2\pi)^{n} \sum_{\gamma'\in \Gamma'} \int_\bbR
    |\p^{n/2}_tF(\gamma',t+\textstyle{\frac i4}|\gamma'|^2)|^2 \, dt \\
& = B (2\pi)^{n}\sum_{\gamma'\in \Gamma'} \sum_{\ell\in\bbZ} |\p^{n/2}_t F(\gamma',
{\textstyle{\frac{\pi}{a}}}\ell+\textstyle{\frac i4}|\gamma'|^2)|^2  \,,
\end{align*}
where the last identity follows from (iv) in Theorem \ref{key2}, the
fact that $PW_a(\bbC)$ is closed under (fractional) differentiation and
the 
classical Whittaker--Kotelnikov--Shannon Theorem.

Conversely, by the
same sequences of equalities, using Corollary \ref{n-dim-version} again,
\begin{align*}
A  (2\pi)^{n}\sum_{\gamma'\in \Gamma'} \sum_{\ell\in\bbZ} |\p^{n/2}_t F(\gamma',
 \textstyle{\frac{\pi}{a}}  \ell+\textstyle{\frac i4}|\gamma'|^2)|^2
& =A  (2\pi)^{n-1}\int_{-a}^0 \sum_{\gamma'\in \Gamma'} |\lambda|^n |\phi(\gamma',\lambda)|^2
e^{-\frac{|\lambda|}{2}{|\gamma'|^2}} 
\, d\lambda \\
& \le  (2\pi)^{n-1}
\int_{-a}^0 \| \phi(\cdot,\lambda)\|_{\cF^\lambda}^2  \, d\lambda \\
& =   \| F\|_{\PW_a^n}^2 \,.
\end{align*}
This proves the sufficient condition  in the case of $\PW_a^n$.

Finally,
let $G\in\PW_a$ be given.  Consider $\widetilde G_0$ and for $\eps>0$
define
$$
\Psi_\eps[z,t] = 
 \frac{1}{(2\pi)^{n+1}}  
\int_{-a}^{-\eps}  \tr \big( \beta_\lambda(\widetilde G_0)\beta_\lambda[z,t]^*\big)
|\lambda|^{n/2}\, d\lambda \,. 
$$
By \eqref{inve-form} it follows that $\Psi_\eps\in L^2(\bbH_n)$ and that 
$$
\beta_\lambda(\Psi_\eps) =
\chi_{[-a,-\eps]}(\lambda)|\lambda|^{-n/2} \beta_\lambda(\widetilde
G_0) \quad \text{and}\quad \beta_\lambda(\Delta^{n/2}\Psi_\eps)
=\alpha
\chi_{[-a,-\eps]}(\lambda) \beta_\lambda(\widetilde
G_0) \,
$$
for some constant $\alpha$, $|\alpha|=1$.
Since $\sF\widetilde G_0 \in  \cH^2 \big([-a,-\eps)\big)$,
Theorem \ref{main-4} implies that $\Psi_\eps$
extends to a function $F_\eps\in \PW_a\cap \PW_a^n$.
Moreover, the sequence
$\{\Delta^{n/2}\Psi_\eps\}$ is a Cauchy sequence in $L^2(\bbH_n)$,
that is, $\{F_\eps\}$ is a Cauchy sequence in $\PW_a^n$. 
Let $F$ be its limit.  It is clear that
$\widetilde F_0=\cI_n \widetilde G_0$, where $\cI_n$ is the inverse of
$\Delta^{n/2}$ (see Proposition \ref{Folland-Riesz-pot}), that is,
$\Delta^{n/2}\widetilde F_0=\widetilde G_0$. 
Therefore, by the first part of the theorem,
\begin{align*}
 A \sum_{\gamma\in\Gamma} |G(\gamma)|^2
& = A \sum_{\gamma \in\Gamma}|\Delta^{n/2}_t F(\gamma)|^2
= A \sum_{\gamma\in\Gamma}|\p^{n/2}_t F(\gamma')|^2 
\le \|F\|_{\PW_a^n}^2 = \| G\|_{\PW_a}^2 \\
  & = \|F\|_{\PW_a^n}^2  \le
B \sum_{\gamma\in\Gamma}|\p^{n/2}_t F(\gamma)|^2 = B\sum_{\gamma \in\Gamma}
   |\Delta^{n/2}_t F(\gamma)|^2  = B 
  \sum_{\gamma\in\Gamma} |G(\gamma)|^2 .
\end{align*}
 This proves the  sufficient condition for $\PW_a$.

In order to prove the necessary condition, we have to show that if $b_{j_0}\le a$ for a
$j_0\in\{1,\dots,n\}$, $\Gamma$ fails to be sampling.  By the previous
Theorem \ref{nec-cond-thm}, if $\Gamma$ is sampling for $\PW_a$, then
$L_{(b_1,\dots,b_n)}$ is sampling for $\cF^a(\bbC^n)$. Since $L_{(b_1,\dots,b_n)}$ is the cartesian
product  of the square lattices $L_{b_1},\dots,L_{b_n}$, $L_{(b_1,\dots,b_n)}$ is
sampling for    $\cF^a(\bbC^n)$ if and only if $L_{b_j}$ is sampling
for $\cF^a(\bbC)$, for $j=1,\dots,n$. But this happens if and only if
$b_j> a$ for $j=1,\dots,n$.   This proves the
theorem.
\epf

As a consequence we have
\begin{cor}
 The space
  $\PW_a^n$ admits a frame of reproducing kernels, namely $\{
  K_\gamma:\, \gamma\in \Gamma\}$, where $\Gamma$ is a lattice as in
  Theorem \ref{main-5}.
\end{cor}
\ms

\section{Final remarks and open questions}\label{Final-remarks}

We believe that the spaces we introduced are worth investigating and arise quite
naturally in our multi-dimensional setting.

The  present work leaves some open questions.  First of all,  it
should be proved a more general version of Theorem \ref{main-5}
 by combining the characterization of sampling sequences for
 the 1-dimensional Paley--Wiener space $PW_a$ and some sufficient
 conditions for sampling sequences for the Fock space $\cF(\bbC^n)$
as in \cite{GL}.

Moreover, in this paper we essentially dealt with the Hilbert case and
we left the case $p\neq2$ for future studies.

This analysis is based on the growth condition on entire function
given by the $p$-integrability of their restriction to submanifold
$b\cU$, that is
the boundary of the Siegel domain $\cU$.  It is certainly possibile to
consider also the growth condition
given by the $p$-integrability of  restrictions to Shilov boundary
of Siegel domains of type II.  In order to extend this theory, it is
likely that the results and techniques developed in \cite{CP,CP1,CP2}
will play an important role.

Finally, our formulas and results suggest that the space $\PW^n_a$
might have a privileged role, as for the case of the Drury--Arveson
space, see \cite{ACMPS} for a study of such space on the Siegel domain
$\cU$. \medskip

{\em Acknowledgement.} We like to thank the anonymous referee for
her/his comments and insightful questions that led to a significant
improvement of this paper; in particular to the addition of Proposition
\ref{referee-prop}, 
Theorem \ref{nec-cond-thm}, and the necessary condition in Theorem  
\ref{main-5}.  \medskip

\bibliography{exp-siegel-bib29-3-21}
\bibliographystyle{plain}

\end{document}